\documentclass[a4paper,english,10pt]{article}
 \usepackage[applemac]{inputenc}
\usepackage{babel}
\usepackage{wrapfig}
\usepackage{epsf}
\usepackage[dvips]{epsfig}
\usepackage{subfig}
 \usepackage{graphicx}
\usepackage{hyperref}
\usepackage{amssymb}
\usepackage{amsmath}
\usepackage[nice]{nicefrac}
\usepackage{amsfonts}
\usepackage{dsfont}
\usepackage[T1]{fontenc}
\usepackage{wasysym}
\usepackage{amssymb}
\usepackage{stmaryrd}
\usepackage{aeguill}
\usepackage{mathrsfs}

\usepackage{tikz,tkz-tab}

\newtheorem{theoreme}{Theorem}[section]
\newtheorem{prop}[theoreme]{Proposition}

\newtheorem{lem}[theoreme]{Lemma}

\newcommand{\finpreuvem}{\quad\square}

\renewcommand{\epsilon}{\varepsilon}
\def\btab{\begin{eqnarray*}}
\def\etab{\end{eqnarray*}}
\def\beq{\begin{equation}}
\def\eeq{\end{equation}}

 \newcommand{ \p }{\mathbb{P} }

 \newcommand{ \E }{\mathbb{E}}

 \newcommand{ \R }{ \mathbb{R} }
 
 \newcommand{\N}{ \mathbb{N} }

\newcommand{ \bp}{ \mathbf{p}}
\newcommand{ \1}{ \mathbf{1}}
\newtheorem{The}{{\bf Theorem}}[section]
\newtheorem{Def}[The]{{\bf Definition}}
\newtheorem{Rem}[The]{{\bf Remark}}

\makeatletter

\@addtoreset{equation}{section} \makeatother

\title{Diseases transmission in a $z$-ary tree 
\author{  P. Debs, T. Haberkorn  \footnote{Laboratoire MAPMO - C.N.R.S. UMR 7349 - F\'ed\'eration Denis-Poisson, Universit\'e d'Orl\'eans
(France). \newline \vspace{0.1cm}  $\quad$  MSC 2000  . \newline \vspace{0.5cm} \textit{Key words: Galton Watson, Dynamical System. } }
}}

\begin{document}

\maketitle

\begin{abstract}
We extend some results of Itai Benjamini and Yuri Lima (see \href{http://arxiv.org/pdf/1305.2610.pdf}{\cite{Benjamini}}). In this paper they consider a binary tree $\mathbb T_n$ of height $n$, each leaf is either infected by one of $k$ diseases or not infected at all. In other words, $x$ at generation $n$ is infected by the $i$-th infection with probability $p_i$ and sane with $p_{k+1}$. Moreover  the infections are independently distributed for each leaf.\\
Infections spread along the tree based on specific rules. In their paper they study the limit distribution of the root of $\mathbb T_n$ as $n$ goes to infinity. \\
Here we want to study the more general case of a Galton-Watson tree and a $z$-ary tree.  
\end{abstract}

\section{Introduction}
First we recall the definition of a Galton Watson tree (GW) and give a few notations.
Assume that $N$ is a $\mathbb N$-valued random variable following a distribution $q$, in other words $\p(N=i)=q_i$ for $i\in\mathbb N$ and to have an interesting problem, we assume that $q_0+q_1=0$ (B\"otcher case).\\ 
Let $\phi$ be the root of the tree and $ N_\phi$ an independent copy of  $ N$. Then, we draw $N_\phi$ children of $\phi$: these individuals are the first generation. In the following we write $N$ for $N_\phi$ for typographical simplicity. At the $n$-th generation, for each individual $x$ we pick $ N_x$ an independent copy of $ N$ where $N_x$ is the number of children of $x$ and so on. The set $\mathbb T$, consisting of the root and its descendants, forms a GW of offspring distribution $q$.\\
We denote by $|x|$ the generation of $x$ and for $n\in\mathbb N$, $\mathbb T_n=\lbrace x\in \mathbb T,|x|\leq n\rbrace$ the GW cut at height $n$ and the leaves of $\mathbb T_n$ are the elements of $\mathbb T_n\backslash \mathbb T_{n-1}$.  \\
According to Neveu's notation (\cite{Neveu}), to each vertex $x$ at generation $m\in\mathbb N$, we associate a sequence $x_1\dots x_m $ where $x_i\in \mathbb N$ and to simplify we write $x=x_1\dots x_m $.\\
This sequence gives the complete ``genealogy'' of $x$:  if $y=x_1\dots x_i$ with $|y|=i<m$, $y$ is the ancestor of $x$ at generation $i$ and we write $y<x$.\\
Note that a $z$-ary tree is just a particular case of a GW with $q_z=1$.\\
In \cite{Benjamini}, Benjamin and Lima consider the spread of an infection in a binary tree denoted $\mathbb T_n$ of height $n$. More precisely, they consider a probability vector $\bp=(\bp_1,\dots, \bp_{k+1})\in\mathbb (\mathbb R_+^*)^{k+1}$ satisfying 
\begin{equation}
\sum_{i=1}^{k+1}\bp_{i}=1
\end{equation}
and each of the nodes of $\mathbb T_n$ is infected or not by one of the disease $\lbrace1,\dots,k\rbrace$ with the following rules 
\begin{itemize}
\item Independently of the others, each leaf is infected according to $\bp$
\[\p(\mbox{leaf is infected by $i$})=\bp_i, \,\p(\mbox{the leaf is not infected})=\bp_{k+1}.\]
\item Nodes at generation $n-1$ are infected this way:
\begin{enumerate}
\item[(R1)] if both children have the same state (infected or not), the ancestor is infected (or not) by it;
\item[(R2)] if both children are infected by different diseases, the ancestor is not infected;
\item[(R3)] if only one of the children is infected, the ancestor is infected by it.  
\end{enumerate}
\item This step is repeated for level $n-2$ and so on.
\end{itemize}
One of their results is the asymptotic behavior of $\bp(n)$, the distribution of the state of the root of $\mathbb T_n$, i.e. the asymptotic behavior of 
\[\forall\, 1\leq i\leq k\  \p(\mbox{root is infected by }i)=\bp_i(n),\,\p(\mbox{root is not infected })=\bp_{k+1}(n). \]
They obtain the following result  ($0_{i}$ denotes $i$ successive zeros)
\begin{The}
Assume that $\bp_1\geq\bp_2\geq\dots\geq\bp_k$.
\begin{enumerate}
\item If $\bp_1=\dots=\bp_k$, then $\bp(n)$ converges to $\left(\frac{1}{2k-1}, \dots,\frac{1}{2k-1},\frac{k-1}{2k-1}\right)$.
\item If $\bp_1=\dots=\bp_i>\bp_{i+1}$ for some $i\in\llbracket1,k-1\rrbracket$, then $\bp(n)$ converges to \\
$\left(\frac{1}{2i-1},\dots,\frac{1}{2i-1},0_{k-i},\frac{i-1}{2i-1}\right)$, where the entry $\frac{1}{2i-1}$ repeats $i$ times. 
\end{enumerate}
\end{The}
The aim of the present paper is to extend, when it can be, the previous results in the case of a $z$-ary tree for $z>2$ and in a very specific case for a GW.
Consider the family of probabilities $\mathscr P_k$ define by
$$\mathscr P_k:=\left\lbrace \mathbf{p}=(\mathbf{p}_1,\dots,\mathbf{p}_{k+1})\in(\mathbb R_+^*)^{k+1} : \sum_{i=1}^{k+1}\mathbf{p}_i=1 \right\rbrace $$
and $(X_j)_{j\geq1}$ i.i.d. random vectors $\lbrace0,1\rbrace^{k}$-valued such that:
$$\p(X_{j}=e_i)=\mathbf{p}_i, \p\left(X_{j}=\1\right)=\mathbf{p}_{k+1} $$
where $(e_i)_{i=1}^{k}$ are the canonical vectors of $\mathbb R^{k}$ and $\1=\sum_{i=1}^ke_i$.\\ 
First we have to give the spread rules for the $k$ diseases in a GW $\mathbb T_n$: 
\begin{itemize}
\item Initially each leaf $x$ of $\mathbb T_n$ is associated to a random variable $X_x$ and is infected as follows: 
\[\p(\mbox{$x$ is infected by $i$})=\p(X_{x}=e_i)=\mathbf{p}_i,
\p(\mbox{$x$ is not infected})=\p(X_{x}=\1)=\mathbf{p}_{k+1}.\]
(Consequently, each leaf is infected i.i.d. according to $\bp$.)
\item Nodes at generation $n-1$ are infected this way:
\begin{enumerate}
\item[(R1')] if all the children have the same state (infected or not) the ancestor is infected (or not) by it;
\item[(R2')] if two children are infected with different diseases, the ancestor is not infected;
\item[(R3')] if some children are infected by a single disease and the others are not infected, the ancestor is infected by it.  
\end{enumerate}
It can be express this way:
$$\mbox{for $|x|=n-1$, }{X}_x=\left\lbrace\begin{array}{cl}
\bigotimes_{i=1}^{N_x}X_{xi}&\mbox{ if ${\Vert \bigotimes_{i=1}^{N_x}X_{xi}\Vert=1} $}\\
\1&\mbox{otherwise,}
\end{array}\right. $$
where $\otimes:\mathbb R^k\times \mathbb R^k\rightarrow\mathbb R^k, (x,\,y)\mapsto x\otimes y=(x_1y_1,x_2y_2,\dots,x_ky_k)$.
\item We repeat this step for level $n-2$ and so on.
\end{itemize}
As claimed, we want to determine $\bp(n)$ the distribution of the state of the root and its asymptotic behavior, in other words the law of $X_\phi$ (or $X_0$) and in the case of $\mathbb T_n$:
$$\forall i\in\llbracket 1,k\rrbracket, \p(X_0=e_i)=\bp_i(n)\mbox{ and } \p(X_0=\1)=\bp_{k+1}(n).$$
In all of our results we
 can assume without loss of generality that $\bp_1\geq \bp_2\geq\dots\geq\bp_k$. 
\begin{The}\label{first}
\begin{enumerate}
\item For $\bp\in \mathscr P_k$ such that $\bp_1=\dots=\bp_k$, if  $\bp(n)$ converges, it does to $(\bar x,\dots,\bar x,1-k\bar x)$ where $\bar x$ is the unique fixed point in $(0,\nicefrac{1}{k}]$ of:
$$f_k(x):=G_N(1-(k-1)x)-G_N(1-kx)$$ 
where $G_N$ is the generating function of $N$. 
\item For $\bp\in \mathscr P_k$ such that $\bp_1=\dots=\bp_i>\bp_{i+1}\geq\bp_{i+1}\geq\dots\geq \bp_k$, then if $\bp(n)$ converges, it does to $(\bar x,\dots,\bar x,0_{k-i},1-i\bar x)$ where $\bar x$ is the unique fixed point of $f_i$ in $(0,\nicefrac{1}{i}]$.
\item If $i=1$, $\bp(n)$ converges to $(1,0_{k})$.
\end{enumerate}
\end{The}
Note that the third point says that if there is only one major disease, regardless of the law of reproduction of $N$, this disease spreads a.s. to the root (asymptotically).\\ 
In what follows, assume that $N=z$ a.s., in other words we have a $z$-ary tree. 

\begin{The}\label{second}
\begin{enumerate}
\item If $z\in\lbrace3,4, 5\rbrace$, for $\bp\in \mathscr P_k$ such that $\bp_1=\dots=\bp_k$, $\bp(n)$ converges to $(\bar x,\dots,\bar x,1-k\bar x)$ where $\bar x$ is the unique fixed point in $(0,\nicefrac{1}{k}]$ of:
\begin{equation}
f_{z,k}(x):=(1-(k-1)x)^z-(1-kx)^z.
\end{equation}
\item For $z\in\lbrace3,4, 5\rbrace$, for $\bp\in \mathscr P_k$ such that $\bp_1=\dots=\bp_i>\bp_{i+1}\geq\bp_{i+1}\geq\dots\geq \bp_k$, $\bp(n)$ converges to $(\bar x,\dots,\bar x,0_{k-i},1-i\bar x)$ where $\bar x$ is the unique fixed point of $f_{z,i}$ in $(0,\nicefrac{1}{i}]$.
\item If $z= 6$,  for $\bp\in \mathscr P_k$ such that $\bp_1=\dots=\bp_k$, $\bar x$ is a repelling point of $f_{z,k}$.
\end{enumerate}
\end{The}

In section \ref{s_z6_i2}, we study completely the case $z=6$ and $i=2$, where $\bp(n)$ does not converge anymore:
\begin{The}\label{concl}
For $z=6$ and $i=2$, denote $\theta:=\lbrace\bar x_\ell, \bar x_r\rbrace$ where $\bar x_\ell$ and $\bar x_r$ are the fixed points of $f_{6,2}^2=f_{6,2}\circ f_{6,2}$ distinct of $\bar x$. Then  for almost all $\bp\in\mathscr P_k$ such that $\bp_1=\bp_2>\bp_3\geq\dots\geq\bp_k$ 
\[\lim_{n\rightarrow\infty} \bp(2n)=(x,x,0_{k-2},1-2x) \]
where $x\in\theta$.
\end{The}
\begin{Rem}
This result is not limited to the case $z=6$ and $i=2$. Indeed, denoting $\hat x_{z,i}=\mathrm{argmax}_{[0,\nicefrac{1}{i}]}f_{z,i}$, the only conditions a case $(z,i)$ has to satisfy are~:
\begin{enumerate}
\item
	$\bar{x}_{z,i}$ is such that $\partial_x f_{z,i}(\bar{x}_{z,i}) < -1$.
\item
	$f_{z,i}(\hat{x}_{z,i}) > \hat{x}_{z,i}$.
\item
	$f_{z,i}(\nicefrac{1}{i}) < \hat{x}_{z,i}$.
\end{enumerate}
\end{Rem}
The paper is structured as follows: Section 2 gives the discrete dynamical system whose study leads to our main Theorems. The same section also gives results for the GW case. Section 3 focuses on $z$-ary tree and concludes the proof of Theorem \ref{second}. Section 4 is devoted to the proof of Theorem \ref{concl}. Finally, Section 5 gives some ideas for extensions of this work. 

\section{General results for a Galton-Watson tree}
A major part of this work consists in the study of discrete dynamical systems: given a function $f$ and a value $x$, we study the behavior of the sequence $f^n(x)$.\\
In this section, we give the studied function $f$ and some global results linked to our problem. \\
To find a recursion formula, assuming that $\phi$ is the root of $\mathbb T_{n+1}$, its children $\phi i$ are root nodes of $N$ independent GW of height $n$. Then the distribution $\bp(n+1)$ of $X_0$, the state of $\phi$, is completely determined by the distribution of $(X_i)_{i=1}^N$, the independent states of its children with distribution  $\bp(n)$. 
\begin{lem}\label{rec}
For all $1\leq i\leq k+1$ and $n\geq1$:
\begin{equation}
\bp_{i}(n+1)=\left\lbrace\begin{array}{cl}
G_{N}(\bp_{k+1}(n)+\bp_i(n))-G_{N}(\bp_{k+1}(n)),&\mbox{ for $i\neq k+1$}\\
1-\sum_{j=1}^k\bp_{i}(n+1),&\mbox{ otherwise.}
\end{array}\right.
\end{equation} 
\end{lem}

\begin{proof}
To simplify our proof denote by $ S_z:= \lbrace 1\leq i\leq z,\, X_i=\1\rbrace$, the non infected sites in a $z$-sized population.
According to (R1')-(R3'):
\begin{align*}
\bp_i(n+1)&=\p\left(\bigotimes_{j=1}^NX_{j}=e_i\right)=\sum_{z=2}^\infty q_z\p\left(\bigotimes_{j=1}^zX_j=e_i\right)=\sum_{z=2}^\infty q_z\sum_{\ell=0}^{z-1}\p\left(|S_z|=\ell, \bigotimes_{i=1}^zX_i=e_i    \right)\\
&=\sum_{z=2}^\infty q_z\sum_{\ell=0}^{z-1}C_{z}^{\ell}\bp_{k+1}^\ell(n) \bp_{i}^{z-\ell}(n)=\sum_{z=2}^\infty q_z((\bp_{k+1}(n)+\bp_i(n))^z-\bp_{k+1}^{z}(n))\\
&=G_{N}(\bp_{k+1}(n)+\bp_i(n))-G_{N}(\bp_{k+1}(n)).\hfill\finpreuvem
\end{align*}
\end{proof}
If we define $F:\mathbb R^{k+1}\rightarrow\mathbb R^{k+1}$ by :
\begin{equation*}
F_{i}(x)= \left\lbrace\begin{array}{ll}
G_{N}(x_{k+1}+x_i)-G_{N}(x_{k+1}),&\mbox{ for $1\leq i\leq k$}\\
1-\sum_{i=1}^kF_{i}(x_1,\dots,x_{k+1}),&\mbox{ otherwise.}\end{array}\right.
\end{equation*}
we see that $\bp(n)=F^n(\bp)$.\\
In fact our problem consists in the study of the fixed points of the function $F$.  
Like in \cite{Benjamini}, we first consider the uniform case assuming that $\bp_1=\dots=\bp_k$. Obviously for all $n\geq1, \,\bp_1(n)=\dots=\bp_k(n)$ implying that we just have to study:
\[\bp_1(n+1)=G_{N}(1-(k-1)\bp_1(n))-G_{N}(1-k\bp_1(n)).\]
For this purpose, define $f_k:(0,\nicefrac{1}{k}]\rightarrow(0,\nicefrac{1}{k}]$ by
\begin{equation}
f_k(x)=G_{N}(1-(k-1)x)-G_{N}(1-kx).
\end{equation}
We obtain the following
\begin{lem}\label{fixed}
$f_k$ admits a unique fixed point in $(0,\nicefrac{1}{k}]$.
\end{lem}
\begin{figure}[h!]
\caption{$k=4$, $q_3=q_6=q_{10}=\frac{1}{3}$} \label{fe}
\begin{center}
\includegraphics[width=8cm]{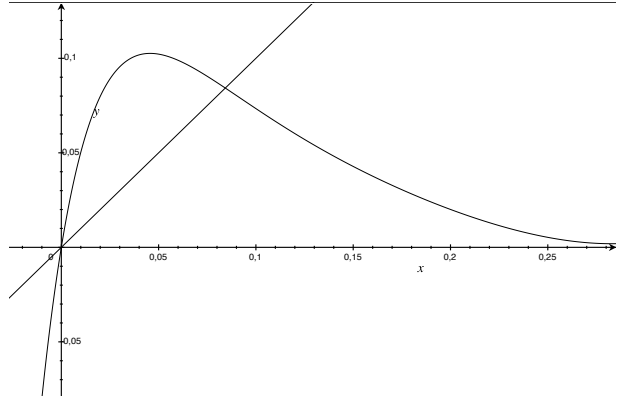}
\end{center}
\end{figure}

\begin{proof}
Properties of generating functions ensure that $f_k\in\mathcal C^\infty((0,\nicefrac{1}{k}])$ and on this interval:
\begin{align*}
f_k^{\prime}(x)&=kG_{N}^\prime(1-kx)-(k-1)G_{N}^\prime(1-(k-1)x)\\
&=k\E[N(1-kx)^{N-1}]-(k-1)\E[N(1-(k-1)x)^{N-1}].
\end{align*}
As $\lim_{x\rightarrow0^+}f_k^\prime(x)=\E[N]\geq 2$, $f_k(0) = 0$, and 
\[f_k(\nicefrac{1}{k})=G_N(\nicefrac{1}{k})-G_N(0)=\sum_{k\geq2}q_zk^{-z}<\nicefrac{1}{k},\]
$f_k$ admits at least one fixed point  $\bar{x}_k\in(0,\nicefrac{1}{k}]$. 
For $x\in(0,\nicefrac{1}{k}]$:
\begin{eqnarray}
f_k(x)=x&\Leftrightarrow&\sum_{z\geq2}q_z\left((1-(k-1)x)^z-(1-kx)^z\right)=x\nonumber\\
&\Leftrightarrow&\sum_{z\geq 2}q_zx\sum_{j=0}^{z-1}(1-(k-1)x)^{z-1-j}(1-kx)^j=x\nonumber\\
&\Leftrightarrow&\sum_{z\geq 2}q_z\sum_{j=0}^{z-1}(1-(k-1)x)^{z-1-j}(1-kx)^j=1\label{fastoche}.
\end{eqnarray}
$x\mapsto \sum_{z\geq 2}q_z\sum_{j=0}^{z-1}(1-(k-1)x)^{z-1-j}(1-kx)^j$ is strictly decreasing on $\left(0,\nicefrac{1}{k}\right]$ and thus bijective. As a result $\eqref{fastoche}$ has at most one solution on $\left(0,\nicefrac{1}{k}\right]$ implying the uniqueness of $\bar x_k$.
\end{proof}\hfill$\finpreuvem$

\begin{lem}\label{zeroplus}
There exists $\eta>0$ such that 
\[\forall n\in \mathbb N, \bp_1(n)\geq \alpha:=\min \lbrace G_N(\eta),\bp_1\rbrace\]
\end{lem}

\begin{proof}
Writing $G_N(1-x)=G_N(1)-xG_N^\prime(1)+\varepsilon(x)$ where $\nicefrac{\varepsilon(x)}{x}\underset{x\rightarrow0}{\longrightarrow}0$, there exists $\eta^\prime>0$ small enough such that if $0<x\leq \eta^\prime$ then $|\varepsilon(x)|\leq \nicefrac{x}{2k}$. Then for $x\leq \nicefrac{\eta^\prime}{k}=:\eta$ and $0<y<(k-i)x$
\begin{eqnarray*}
G_N(1-(i-1)x-y)-G_N(1-ix-y)&=&xG_N^\prime(1)+\varepsilon((i-1)x+y)-\varepsilon(ix+y)\\
&\geq&xG_N^\prime(1)-\frac{ix+y}{k}\geq G_N^\prime(1)-x\geq 2x-x=x,
\end{eqnarray*}
as $G_N^\prime(1)=\E[N]\geq2$.\\
If $x>\eta$, using the fact $G_N(1-(i-1)x-y)-G_N(1-ix-y)$ is decreasing in $y$ 
\begin{eqnarray*}
\inf_{x\in[\eta,\nicefrac{1}{k}]}\inf_{y\in[0,(1-ix)\wedge (k-i)x]}G_N(1-(i-1)x-y)-G_N(1-ix-y)\\
\geq \inf_{x\in[\eta,\nicefrac{1}{k}]}\inf_{y\in[0,1-ix]}G_N(1-(i-1)x-y)-G_N(1-ix-y)=\inf_{x\in[\eta,\nicefrac{1}{k}]}G_N(x)=G_N(\eta)>0.
\end{eqnarray*}
Consequently an obvious recurrence gives $\bp_1(n)\geq \alpha=\min\lbrace G_N(\eta), \bp_1\rbrace$ for all $n\geq0$.
\end{proof}\hfill$\finpreuvem$\\
\begin{Rem}
The two previous lemmas induce that, in the uniform case, if $\bp(n)$ converges, it does to $(\bar x_k,\dots,\bar x_k,1-k\bar x_k)$, which is the first point of Theorem \ref{first} . 
\end{Rem}

The following lemma ensures that the ``minor" diseases can not spread to the root asymptotically and thus the second part or Theorem \ref{first}:  
\begin{lem}\label{zero}
In the non uniform case with $\bp_1=\dots=\bp_i>\bp_{i+1}$, for all $j>i$, $\bp_j(n)\underset{n\rightarrow\infty}{\rightarrow}0$. 
\end{lem}
\begin{proof}
Note that we just have to prove that $\lim_{n\rightarrow+\infty}\bp_{i+1}(n)=0$. Writing $w_n=\frac{\bp_{i+1}(n)}{\bp_1(n)}$:
\begin{equation}\label{dec}
\forall n\geq 0,\,w_{n+1}=w_n\frac{\sum_{z\geq2}q_z\sum_{j=0}^{z-1}(\bp_{i+1}(n)+\bp_{k+1}(n))^{z-1-j} (\bp_{k+1}(n))^j }{\sum_{z\geq2}q_z\sum_{j=0}^{z-1}(\bp_{1}(n)+\bp_{k+1}(n))^{z-1-j} (\bp_{k+1}(n))^j}\leq w_n
\end{equation}
as $\bp_1(n)\geq \bp_{i+1}(n)$. Thus, $(w_n)$ is a positive non increasing sequence, and consequently converges. Denote by $\ell=\lim_{n\rightarrow+\infty}w_n$ and as $w_0<1$, note that $\ell<1$.  \\
We can  find a subsequence ${n_m}$ such that $\lim_{m\rightarrow+\infty}\bp_{j}(n_m)=a_j$ for all $j\leq k+1$. From lemma \ref{zeroplus}, we have $a_1>0$. Now, assume that $a_{i+1}>0$. Since $\ell<1$, we have $a_1>a_{i+1}>0$. Using \eqref{dec}
\begin{eqnarray*}
\ell=\lim_{m\rightarrow+\infty}w_{n_m+1}&=&\lim_{m\rightarrow+\infty}w_{n_m}\frac{\sum_{z\geq2}q_z\sum_{j=0}^{z-1}(\bp_{i+1}(n_m)+\bp_{k+1}(n_m))^{z-1-j} (\bp_{k+1}(n_m))^j }{\sum_{z\geq2}q_z\sum_{j=0}^{z-1}(\bp_{1}(n_m)+\bp_{k+1}(n_m))^{z-1-j} (\bp_{k+1}(n_m))^j}\\
&=&\ell \frac{\sum_{z\geq2}q_z\sum_{j=0}^{z-1}(a_{i+1}+a_{k+1})^{z-1-j} (a_{k+1})^j }{\sum_{z\geq2}q_z\sum_{j=0}^{z-1}(a_{1}+a_{k+1})^{z-1-j} (a_{k+1})^j}<\ell
\end{eqnarray*}
 which is a contradiction. Then $a_{i+1}=0$ and consequently $\bp_{i+1}(n)\underset{n\rightarrow\infty}{\rightarrow}0$. \hfill$\finpreuvem$
\end{proof}

The two previous lemmas have an important role in the following sections but we can also prove easily the point 3 of Theorem \ref{first}. According to lemma
\ref{zeroplus} $\liminf_{n\rightarrow\infty}\bp_1(n)>0$ and from lemma \ref{zero}, $\lim_{n\rightarrow\infty}\bp_{j}(n)=0$ for all $j>1$, then:
 \begin{eqnarray*}
\bp_{1}(n+1)&=&G_N\left(1-\sum_{j=2}^k\bp_j(n)\right)-G_N\left(1-\sum_{j=2}^k\bp_j(n)-\bp_{1}(n)\right)\\
\liminf_{n\rightarrow\infty}\bp_{1}(n+1)&=&G_N(1)-G_N\left(1-\liminf_{n\rightarrow\infty}\bp_{1}(n)\right)=1-G_N\left(1-\liminf_{n\rightarrow\infty}\bp_{1}(n)\right).
\end{eqnarray*} 
Thus, $\liminf_{n\rightarrow\infty}\bp_{1}(n)$ is a fixed point of $x\mapsto1-G_N(1-x)$ on $(0,1]$. This function being strictly increasing, the solution is 1. As a result $\liminf_{n\rightarrow\infty}\bp_{1}(n)=\lim_{n\rightarrow\infty}\bp_1(n)=1$. 


\section{Results for $z$-ary tree}\label{s_zary}

In this section, we investigate the special case where the GW is a $z$-ary tree for $z \geq 3$ ($z=2$ is the case studied in \cite{Benjamini}). In this case, the function $F:\R^{k+1}\to \R^{k+1}$ is:
\begin{equation*}
F_i(x)  = \left\{\begin{array}{ll}
 \left(1-\sum_{j=1,j\neq i}^k x_j\right)^z - \left(1-\sum_{j=1}^k x_j\right)^z,&\ \text{for } 1 \leq i \leq k\\
  1 - \sum_{i=1}^k F_i(x), &\mbox{ otherwise.}
\end{array}\right.
\end{equation*}

When studying the asymptotic behavior of $\bp(n)$, with $\bp_1 = \dots = \bp_i > \bp_{i+1} \geq \dots \geq \bp_k$, it is enough to study the asymptotic behavior of $\tilde \bp(n) = (\bp_1(n),\bp_{i+1}(n),\dots,\bp_k(n)) \in [0,1]^{k-i+1}$. Indeed, we have $\bp_1(n)=\dots=\bp_i(n)$ and $\bp_{k+1}(n) = 1-\sum_{j=1}^k \bp_j(n)$, and thus, if $\tilde\bp(n)$ tends toward $\bar{\tilde{\bp}}$, then $\bp(n)$ tends toward $\bar{\bp} = (\bar{\tilde{\bp}}_1,\dots,\bar{\tilde{\bp}}_1,\bar{\tilde{\bp}}_2,\dots,\bar{\tilde{\bp}}_{k-i+1},1-i\bar{\tilde{\bp}}_1 - \sum_{j=2}^{k-i+1} \bar{\tilde{\bp}}_j)$.

In the uniform case, where $\bp_1(n) = \dots = \bp_k(n)$, lemma \ref{fixed} applies and states that there exists a unique fixed point $(\bar{x}_{z,k},\dots,\bar{x}_{z,k},1-k\bar{x}_{z,k})$ of $F$. Furthermore, in the uniform case, we can restrict our study to the discrete scalar dynamical system whose dynamics is given by $f_{z,k}(x) = F_j(x,\cdots,x,1-kx)$ for $j = 1,\cdots,k$ and $x \in [0,\nicefrac{1}{k}]$, that is
\begin{equation}
f_{z,k}(x) = (1-(k-1)x)^z-(1-kx)^z
\end{equation}

\begin{Rem}
In this section, some of the proofs use differentiation with respect to the integers $z$ or $k$. This has to understood as a differentiation with respect to a relaxation of $z$ or $k$ in $\mathbb R$.  
\end{Rem}

The following lemma gives a lower and upper bound of $\bar{x}_{z,k}$, the unique fixed point of $f_{z,k}$ in $(0,\nicefrac{1}{k}]$.
\begin{lem}\label{framing}
The unique fixed point $\bar{x}_{z,k}$ of $f_{z,k}$ in $(0,\nicefrac{1}{k}]$ satisfies
$$\tilde{x}_{z,k} < \bar{x}_{z,k} < \tilde{x}_{z,k-1},\, \forall z \geq 2, \forall\, k \geq 2,$$
where 
$$\tilde{x}_{z,k} = \frac{1}{k}\left( 1 - \left( \frac{1}{z}\right)^{\frac{1}{z-1}} \right).$$
\end{lem}

\begin{proof}
Recall that according to the proof of lemma \ref{fixed}, $\Delta_{z,k}(x) = f_{z,k}(x)-x$ only has one zero $\bar{x}_{z,k}$ in $(0,\nicefrac{1}{k}]$. This function is positive on $(0,\bar{x}_{z,k})$ and negative on $(\bar{x}_{z,k},\nicefrac{1}{k}]$ (see Figure \ref{TabVarF}).

\paragraph{Lower bound:} We prove that, for all $z \geq 2$, and $k \geq 1$, $\Delta_{z,k}(\tilde{x}_{z,k}) > 0$ which implies that $\bar{x}_{z,k} > \tilde{x}_{z,k}$.  $\Delta_{z,k}(\tilde{x}_{z,k})$ writes as:
\begin{equation*}
\Delta_{z,k}(\tilde{x}_{z,k}) = \left( 1 - \frac{k-1}{k} (1-a_z) \right)^z - a_z^z - \frac{1}{k}(1-a_z),
\end{equation*}
where $a_z = (\nicefrac{1}{z})^{\nicefrac{1}{(z-1)}}$, so that $\tilde{x}_{z,k} = (1-a_z)/k$. Obviously, $\Delta_{z,k}(\tilde{x}_{z,k})$ goes to $0$ as $k$ tends to infinity. Differentiating $\Delta_{z,k}(\tilde{x}_{z,k})$ with respect to $k$:
\begin{eqnarray*}
\frac{d \Delta_{z,k}(\tilde{x}_{z,k})}{dk} & = & z \left( 1 - \frac{k-1}{k} (1-a_z) \right)^{z-1} \left(-\frac{1}{k^2}\right) (1-a_z) + \frac{1}{k^2} (1-a_z)\\
& = & \frac{\tilde{x}_{z,k}}{k} ( 1 - z(1-(k-1)\tilde{x}_{z,k})^{z-1})\\
& < & \frac{\tilde{x}_{z,k}}{k} ( 1 - z(1-k\tilde{x}_{z,k})^{z-1}) = 0
\end{eqnarray*}
where we used the fact that $z(1-k\tilde{x}_{z,k})^{z-1} = 1$. So $\Delta_{z,k}(\tilde x_{z,k})$ is decreasing with $k$ and its limit as $k$ tends towards infinity is zero. It is thus positive for all $k \geq 1$ and $z \geq 2$. This concludes the proof that $\bar{x}_{z,k} > \tilde{x}_{z,k}$.

\paragraph{Upper bound:} Similarly to the lower bound, we prove that  for all $z \geq 2$, and $k \geq 2$, $\Delta_{z,k}(\tilde{x}_{z,k-1})$ $ < 0$ which implies that $\bar{x}_{z,k} < \tilde{x}_{z,k-1}$.
As before, $\Delta_{z,k}(\tilde{x}_{z,k-1})$ goes to $0$ as $k$ tends to infinity.

Differentiating $\Delta_{z,k}(\tilde{x}_{z,k-1})$ with respect to $k$:
\begin{eqnarray*}
\frac{d \Delta_{z,k}(\tilde{x}_{z,k-1})}{dk} & = & -z\left( 1 - \frac{k}{k-1}(1 - a_z)^{z-1}\right) \frac{(1-a_z)}{(k-1)^2} + \frac{(1-a_z)}{(k-1)^2}\\
& = & \frac{(1-a_z)}{(k-1)^2} \left( 1 - z\left(1-\frac{k}{k-1}(1-a_z)\right)^{z-1}\right)\\
& > & \frac{(1-a_z)}{(k-1)^2} ( 1 - z(1-k\tilde{x}_{z,k} )^{z-1}) = 0.
\end{eqnarray*}
So $\Delta_{z,k}(\tilde{x}_{z,k-1})$ is increasing and its limit, as $k$ tends towards infinity is $0$. We conclude that $\Delta_{z,k}(\tilde{x}_{z,k-1})$ is negative for all $z \geq 2$ and $k \geq 2$. And in turn we obtain $\bar{x}_{z,k} < \tilde{x}_{z,k-1}$.
\end{proof}\hfill$\finpreuvem$\\

In addition, to the framing of the fixed point of $f_{z,k}$, the proof of lemma \ref{fixed} gives that $f_{z,k}$ has a unique critical point (maximum) in $(0,\nicefrac{1}{k}]$. We denote by $\hat{x}_{z,k}$ this maximum, a direct computation of the critical point of $f_{z,k}$ in $(0,\nicefrac{1}{k}]$ yields the value \eqref{maximum_inflexion}. Furthermore, another direct computation of the zeros of $\partial^2_{x} f_{z,k}$ gives the unique inflection point $x^\star_{z,k}$ of $f_{z,k}$ in $(0,\nicefrac{1}{k}]$:
\begin{equation}\label{maximum_inflexion}
\hat{x}_{z,k} = \frac{k^\frac{1}{z-1}-(k-1)^\frac{1}{z-1}}{k^\frac{z}{z-1}-(k-1)^\frac{z}{z-1}},\quad
x^\star_{z,k} = \frac{k^\frac{2}{z-2}-(k-1)^\frac{2}{z-2}}{k^\frac{z}{z-2}-(k-1)^\frac{z}{z-2}}
\end{equation}
Note that $\hat{x}_{z,k} < x^\star_{z,k}$.

Now, we show that having $\bar{x}_{x,k}$ as an attracting fixed point of $f_{z,k}$ guaranties that the asymptotic behavior of the diseases spread is as stated in our main result.
\begin{prop}\label{CV_attractive}
For the uniform case, if $\bar{x}_{z,k}$ is such that $| \partial_x f_{z,k}(\bar{x}_{z,k})| < 1$, then $\bp(n)$ converges to $(\bar{x}_{z,k},\dots,\bar{x}_{z,k},1-k\bar{x}_{x,k})$.
\end{prop}
\begin{proof}
Since our result does not depend on $z$ and $k$, we will drop the indexation of $f$, $\bar{x}$ and $\hat{x}$ by those integers.\\
Note that if $\hat x \geq \bar x$, necessarily $f^\prime(\bar x)\in[0,1)$ and $f((0,\nicefrac{1}{k}])\subset(0,f(\hat x)]\subset (0,\hat x]$. Since $f((0,\bar x])=(0,\bar x]$ and $f(x)\geq x$ on this interval, $f^n(x)$ tends toward $\bar x$. This result is still true if $x\in[\bar x, \hat x]$ with a similar reasoning.\\
In the rest of the proof we assume that $\hat x<\bar x$ and it is stuctured as follows. First, we study the variations of $f^2 = f\circ f$ and then show that $f^2$ only has one fixed point. Finally we conclude on the convergence of the discrete scalar dynamical system.

{\bf Variations and concavity of $\boldsymbol{f^2}$:}\\
First note that as $\hat x<\bar x$, $f(\hat{x}) > \hat{x}$ and since $f$ is increasing on $[0,\hat{x}]$  there exists a unique $\hat{x}_\ell \in (0,\hat{x})$ such that $f(\hat{x}_\ell) = \hat{x}$. Similarly, $f$ is decreasing in $[\hat{x},\nicefrac{1}{k}]$, if $f(\nicefrac{1}{k}) \leq \hat{x}$ there exists $\hat{x}_r \in (\hat{x},\nicefrac{1}{k}]$ such that $f(\hat{x}_r) = \hat{x}$ and none otherwise.\\
Thus, as $(f^2)^\prime(x)=f^\prime(f(x)) f^\prime(x)$ and $f([0,\nicefrac{1}{k}]) = [0,f(\hat{x})] \subset [0,\nicefrac{1}{k}]$, the only critical points of $f^2$ are $\hat{x}_\ell$, $\hat{x}$ and possibly $\hat{x}_r$. Since $f^2(\hat{x}_\ell) = f(\hat{x})$, it is a maximum of $f^2$ (in $[0,\nicefrac{1}{k}])$, like $\hat x_r$ in case of existence, and thus $\hat{x}$ is a local minimum. To conclude, $f^2$ is increasing on $[0,\hat{x}_\ell]\cup[\hat{x},\min(\nicefrac{1}{k},\hat{x}_r)]$ and decreasing elsewhere. In particular, $f^2$ is increasing on $[\hat{x},\bar{x}] \subset [\hat{x},\min(\nicefrac{1}{k},\hat{x}_r)]$ (we assume that if $\hat x_r$ does not exist in $[0,\nicefrac{1}{k}]$, we take $\hat x_r=+\infty$).\\
Note that this study of the variations of $f^2$ is not restricted to the attracting fixed point case, as long as we have $\hat{x} < \bar{x}$.

A study of the second derivative of $f^2$ gives that it is negative on $[0,\hat{x}_\ell]$ and cancels only once on $[\hat{x},\hat{x}_r]$. So $f^2$ is concave on $[0,\hat{x}_\ell]$ and have exactly one inflection point on $[\hat{x},\hat{x}_r]$ where it is convex then concave. Note that this study of the inflection points of $f^2$ is inconclusive on $[\hat{x}_\ell,x^\star_\ell]$ (with $x^\star_\ell$ the first pre-image of $x^\star$ by $f$) and $[x^\star,\nicefrac{1}{k}]$.

{\bf $\boldsymbol{\bar{x}}$ is the unique fixed point of $\boldsymbol{f^2}$ in $\boldsymbol{(0,\nicefrac{1}{k}]}$:}\\
We reason by contradiction with the maximum number of inflection points of $f^2$ in $[0,\hat{x}_\ell]$ and $[\hat{x},\hat{x}_r]$.

First note that as $\bar{x}$ is such that $f^\prime(\bar{x}) \in (-1,1)$, $(f^2)^\prime(\bar{x}) \in (0,1)$ and so there exists $\eta > 0$ such that $f^2(x) > x$ if $x \in (\bar{x}-\eta,\bar{x})$ and $f^2(x) < x$ if $x \in (\bar{x},\bar{x}+\eta)$. 

Now, assume $\bar x_1=\sup\lbrace x\in (0,\bar x), f^2(x)=x\rbrace$ exists (necessarily $\bar{x}_1 \neq \bar{x}$), thus  $f^2(x) > x$ on $(\bar{x}_1,\bar{x})$ and $(f^2)^\prime(\bar x_1)\geq1$. We face two cases: either $\bar{x}_1 > \hat{x}$, either $\bar{x}_1 < \hat{x}$ (as $\bar{x}_1 = \hat{x}$ gives directly a contradiction since $(f^2)^\prime(\hat{x}) = 0$).

Assume $\bar{x}_1 > \hat{x}$: as $(f^2)^\prime(\bar{x}_1) \geq 1$, $(f^2)^\prime(\hat{x})=0$ and $(f^2)^\prime(\bar{x}) < 1$, $f^2$ has at least one inflection point in $(\hat{x},\bar{x})$. Since $f^2(\bar{x}_1) = \bar{x}_1$, then $\bar x_2=f(\bar{x}_1) =\inf\lbrace x\in(\bar x,\bar x_r), f^2(x)=x\rbrace$ is also a fixed point of $f^2$. 
Similarly as before, $f^2$ has at least one inflection point in $(\bar{x},\hat{x}_r)$. So $f^2$ has a least 2 inflection points on $[\hat{x},\hat{x}_r]$ which raises a contradiction.

Since $f^2$ has no fixed point in $[\hat{x},\bar{x})$, it has none in $[\hat{x}_\ell,\hat{x}]$ because $f^2$ is decreasing on this interval. Moreover, $f^2$ is concave on $[0,\hat{x}_\ell]$, $f^2(0) = 0$ and $f^2(\hat{x}_\ell) > \hat{x}_\ell$ so that $f^2(x) > x$ on $(0,\hat{x}_\ell]$.

We thus conclude that $\bar{x}_1$ does not exist. If there exists a fixed point $\bar{x}_2$ of $f^2$ in $(\bar{x},\nicefrac{1}{k}]$ then $f(\bar{x}_2)\in(0,\bar{x})$ is also a fixed point of $f^2$ which is not possible. So $f^2$ doesn't have a fixed point in $(\bar{x},\nicefrac{1}{k}]$ either.
 

\begin{Rem}
Note that a corollary to the uniqueness of the fixed point $\bar{x}$ is that $f^2(\hat{x}) > \hat{x}$.
\end{Rem}


{\bf $\boldsymbol{\bar{x}}$ is asymptotically stable for $\boldsymbol{f}$ in $\boldsymbol{(0,\nicefrac{1}{k}]}$:}\\
We assume that $\hat{x}_r$ exists, otherwise the reasoning is the same with $\hat{x}_r$ replaced by $\nicefrac{1}{k}$ and one less step: we would not have to consider the interval $[\hat{x}_r,\nicefrac{1}{k}]$.


From the variations of $f^2$, we have that $f^2([\hat{x},\bar{x}]) = [f^2(\hat{x}),\bar{x}] \subset [\hat{x},\bar{x}]$. So if $x \in [\hat{x},\bar{x}]$ then $f^{2n}(x)$ is an increasing bounded above sequence, and so it tends towards $\bar{x}$, the unique fixed point of $f^2$ in $[\hat x,\bar x]$. As $\bar{x}$ is the fixed point of $f$ and $f$ is continuous, $f\circ f^{2n}(x) = f^{2n+1}(x)$ also tends towards $\bar{x}$.

Since $f^2(\hat{x}_r) =f(\hat x)< \hat{x}_r$, we have that $f^2([\bar{x},\hat{x}_r]) = [\bar{x},f(\hat{x})] \subset [\bar{x},\hat{x}_r]$. So if $x \in [\bar{x},\hat{x}_r]$, then $f^{2n}(x)$ is a decreasing bounded below sequence. As for the $[\hat{x},\bar{x}]$ case, we conclude that $f^n(x)$ tends toward $\bar{x}$.

As $f^2(\hat{x}) > \hat{x}$ and $f^2(\hat{x}_\ell) = f(\hat{x}) < \hat{x}_r$, we have that $f^2([\hat{x}_\ell,\hat{x}]) = [f^2(\hat{x}),f(\hat{x})] \subset [\hat{x},\hat{x}_r]$. And again, we conclude that if $x \in [\hat{x}_\ell,\hat{x}]$, then $f^n(x)$ tends towards $\bar{x}$.

We have that $f^2((0,\hat{x}_\ell]) = (0,f(\hat{x})]$ and if $x \in (0,\hat{x}_\ell]$, then $f^2(x) > x$ so $f^{2n}(x)$ is an increasing bounded above sequence so it converges to $\bar{x}$. Finally, $f^2([\hat{x}_r,\nicefrac{1}{k}]) = [f^2(\nicefrac{1}{k}),f(\hat{x})] \subset (0,f(\hat{x})]$ and the previous case gives the convergence to $\bar{x}$.

So we have that for all $x \in (0,\nicefrac{1}{k}]$, $\lim_{n\to\infty} f^n(x) = \bar{x}$, which concludes the proof.
\end{proof}\hfill$\finpreuvem$

Figure \ref{TabVarF} sums-up what we gather of the variations of $f_{z,k}$ when $\bar{x}_{z,k}$ is linearly attracting.
\begin{figure}[h!]
\caption{\label{TabVarF} Table of variations of $f_{z,k}$ when $\hat x_{z,k}<\bar x_{z,k}$}
\vspace*{-0.5cm}
\begin{center}
\includegraphics[width=10cm]{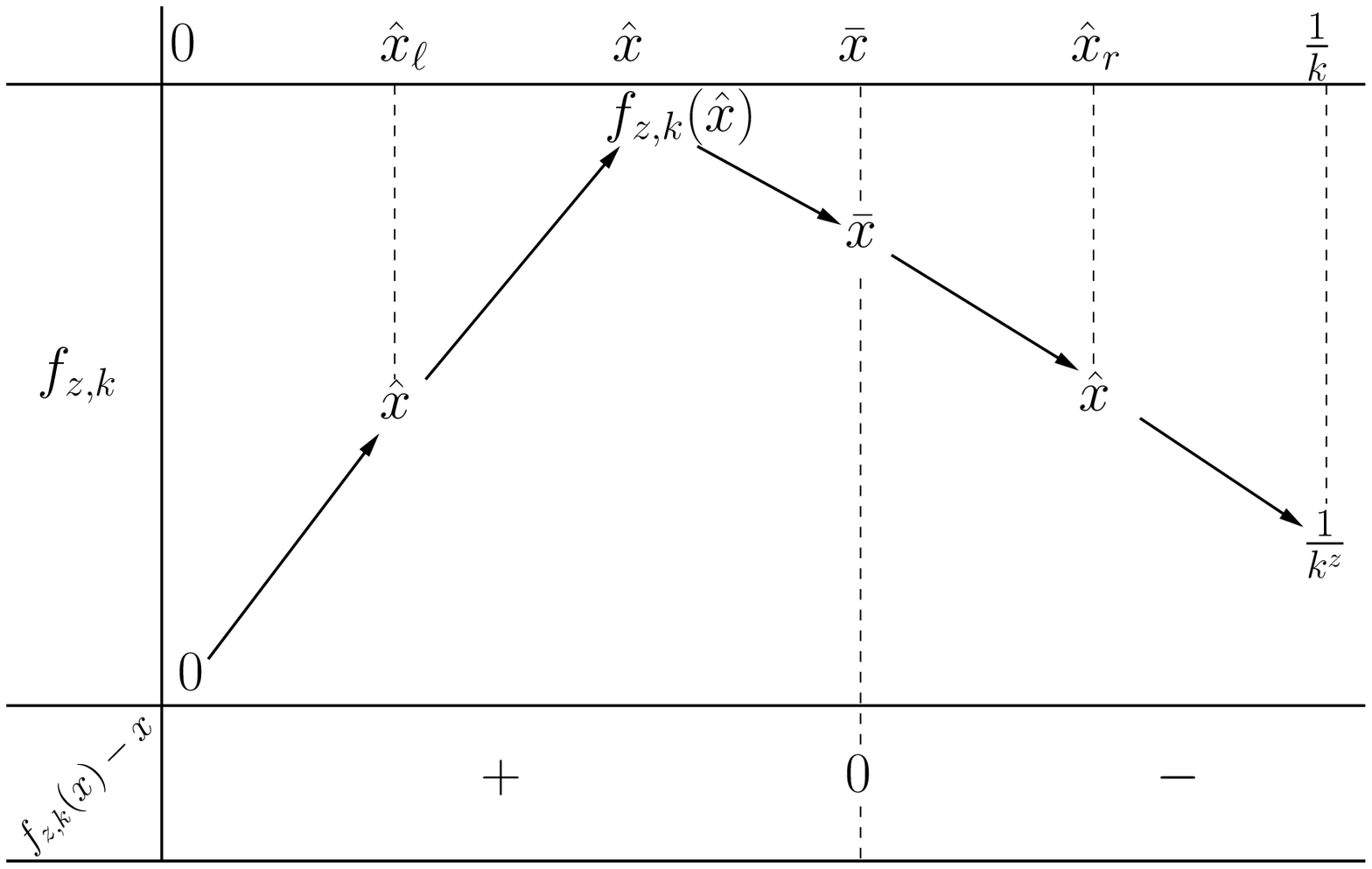}
\end{center}
\end{figure}

\begin{prop}\label{dim1attrac}
If $\bar x_{z,i}$ is a linear attractor for the uniform case with $i$ diseases, then ${\bf{\bar x}}=(\bar x_{z,i},\dots,\bar x_{z,i},0_{k-i},1-i\bar x_{z,i})$ is a linear attractor for the non uniform case with $\bp_1=\dots=\bp_i>\bp_{i+1}\geq\dots\geq \bp_k$. 
\end{prop}
\begin{proof}
${\bf{\bar x}}$ is a linear attractor if all the eigenvalues of the matrix $A:=\frac{\partial \tilde F}{\partial x}({\bf{y}})$ are in $(-1,1)$, with $\tilde{F} = (\tilde{F}_1,\dots,\tilde{F}_{k-i+1}) = (F_1,F_{i+1},\dots, F_k)$, a truncated version of $F$ and ${\bf{y}}=(\bar x_{z,i},0_{k-i})$. $\tilde F$ is defined by
\[ \tilde F_\ell(x)=\left\lbrace\begin{array}{ll}\left(1-(i-1)x_1-\sum_{j=2}^{k-i+1}x_j\right)^z-\left(1-ix_1-\sum_{j=2}^{k-i+1}x_j\right)^z&\mbox{ if $\ell=1$}\\
\left(1-ix_1-\sum_{j=2,j\neq\ell}^{k-i+1}x_j\right)^z-\left(1-ix_1-\sum_{j=2}^{k-i+1}x_j\right)^z&\mbox{ otherwise.}
\end{array}\right.\]

Note that $\frac{\partial \tilde F_m}{\partial x_\ell}({{x}})$ equals to
\[\left\lbrace\begin{array}{ll}
-z(i-1)\left(1-(i-1)x_1-\sum_{j=2}^{k-i+1}x_j\right)^{z-1}+zi\left(1-ix_1-\sum_{j=2}^{k-i+1}x_j\right)^{z-1}&\mbox{ if $m=\ell=1$}\\
-iz\left(1-ix_1-\sum_{j=2,j\neq\ell}^{k-i+1}x_j\right)^{z-1}+iz\left(1-ix_1-\sum_{j=2}^{k-i+1}x_j\right)^{z-1}&\mbox{if $m\neq 1$, $\ell=1$}\\
z\left(1-ix_1-\sum_{j=2}^{k-i+1}x_j\right)^{z-1}&\mbox{if $m=\ell>1$}\\
-z\left(1-ix_1-\sum_{j=2,j\neq\ell}^{k-i+1}x_j\right)^{z-1}+z\left(1-ix_1-\sum_{j=2}^{k-i+1}x_j\right)^{z-1}&\mbox{if $1\neq m\neq \ell\neq 1$}\\
\end{array}\right.\]
implying
\[\frac{\partial \tilde F_m}{\partial x_\ell}({\bf{y}})=\left\lbrace\begin{array}{ll}
-z(i-1)(1-(i-1)\bar x_{z,i})^{z-1}+zi(1-i\bar x_{z,i})^{z-1}&\mbox{ if $m=\ell=1$}\\
0&\mbox{if $m\neq 1$, $\ell=1$}\\
z(1-i\bar x_{z,i})^{z-1}&\mbox{if $m=\ell>1$}\\
0&\mbox{if $1\neq m\neq \ell\neq 1$}\\
\end{array}\right.\]
Thus $A$ is upper triangular and its spectrum is $\lbrace-z(i-1)(1-(i-1)\bar x_{z,i})^{z-1}+zi(1-i\bar x)^{z-1},z(1-i\bar x_{z,i})^{z-1}\rbrace$. First note that as $\bar x_{z,i}$ is an attracting fixed point for the uniform case then $\frac{\partial\tilde F_1}{\partial x_\ell}({\bf{y}})\in(-1,1)$.

It remains to prove that (for $j =2,\dots,k-i+1$),
\begin{equation*}
0 \leq \frac{\partial\tilde F_j}{\partial x_j} (\bar{x}_{z,i},0_{k-i}) = z(1-i\bar{x}_{z,i})^{z-1} < 1,
\end{equation*}
and note that positivity is obvious.

From lemma \ref{framing}, $\bar{x}_{z,i} > \tilde{x}_{z,i}$ for all $z \geq 2$ and $i \geq 1$, and 
\begin{equation*}
\frac{\partial\tilde  F_j}{\partial x_j} (\bar{x}_{z,i},0_{k-i}) < \frac{\partial\tilde F_j}{\partial x_j} (\tilde{x}_{z,i},0_{k-i}) = 1,
\end{equation*}
which concludes the proof.\hfill$\finpreuvem$
\end{proof}

\begin{lem}\label{fixed_limit}
For the uniform case, and $z \geq 2$, there exists $K_z > 0$ such that, for all $k \geq K_z$, the fixed point $(\bar{x}_{z,k},\dots,\bar{x}_{z,k}, 1-k\bar x_{z,k})$ is {\bf attracting} if $z \leq 5$, {\bf repelling} otherwise.
\end{lem}

\begin{proof}
It is enough to study the asymptotic behavior of $\partial_x f_{z,k}(\bar{x}_{z,k})$, the non zero eigenvalue of the linearized dynamical system.

From lemma \ref{framing} and the continuity of $\partial_x f_{z,k}$:
\begin{equation*}
\lim_{k\to\infty} \partial_x f_{z,k} (\bar{x}_{z,k}) = \lim_{k\to\infty} \partial_x f_{z,k}(\tilde{x}_{z,k})
\end{equation*}
Furthermore, the expression of $\tilde{x}_{z,k}$ gives us an equivalency as $k$ tends to infinity
\begin{eqnarray*}
\partial_x f_{z,k}(\tilde{x}_{z,k}) & = & -z(k-1)(1-(k-1)\tilde{x}_{z,k})^{z-1} + zk(1-k\tilde{x}_{z,k})^{z-1}\\
& = & -z(k-1)(1-(k-1)\tilde{x}_{z,k})^{z-1} + k\\
& = & k(1-z(1-(k-1)\tilde{x}_{z,k})^{z-1}) + z(1-(k-1)\tilde{x}_{z,k})^{z-1}\\
& \underset{k\rightarrow\infty}{\sim} & k\left( 1 - z(k-k\tilde{x}_{z,k})^{z-1}\left( 1 + \frac{\tilde{x}_{z,k}}{1-k\tilde{x}_{z,k}}\right)\right) + 1\\
& \underset{k\rightarrow\infty}{\rightarrow} &  1 - (z-1)z^\frac{1}{z-1} \left( 1 - \left(\frac{1}{z}\right)^\frac{1}{z-1}\right)
\end{eqnarray*}
And this limit is in $(-1,1)$ for $z \leq 5$ and is less than $-1$ for $z \geq 6$, which concludes the proof.
\end{proof}\hfill$\finpreuvem$

\begin{Rem}
If we were able to prove that the eigenvalue $\partial_x f_{z,k}(\bar{x}_{z,k})$ is decreasing with respect to $k$, then its asymptotic value would directly give us that for $z \leq 5$, the fixed point is a linear attractor for all $k\geq1$. Even though this monotonicity seems true numerically, we were not able to prove it, and were reduced to cumbersome computations for $z \in \{5,6\}$.
\end{Rem}

The following result focuses on the fixed points of 3-ary, 4-ary and 5-ary trees:
\begin{lem}\label{fixed345_attrac}
For the uniform case, if $z \in \{3,4,5\}$, and $k\geq 2$ the fixed point $\bar x_{z,k}$ is linearly attracting, that is $\partial_x f_{z,k}(\bar{x}_{z,k}) \in (-1,1)$.
\end{lem}
\begin{proof}
For the upper bound on the derivative, we have, for all $z \geq 2$
\begin{eqnarray*}
\partial_x f_{z,k}(\bar{x}_{z,k}) & = & -z(k-1)z(1-(k-1)\bar{x}_{z,k})^{z-1} + kz(1-k\bar{x}_{z,k})^{z-1}\\
& < & -z(k-1)(1-k\bar{x}_{z,k})^{z-1} + kz(1-k\bar{x}_{z,k})^{z-1}\\
& < & z(1-k\bar{x}_{z,k})^{z-1}\\
& < & z(1-k\tilde{x}_{z,k})^{z-1} = 1,
\end{eqnarray*}
the last inequality comes from lemma \ref{framing}.

%

%
%
%
For the lower bound on the derivative, we treat differently the cases $z \in \{3,4\}$ and $z = 5$.

{\bf Lower bound, case $\mathbf{z \in \{3,4\}}$:} noting that the unique inflection point of $f_{z,k}$ is such that $x^\star_{z,k} > \hat{x}_{z,k}$
\begin{equation*}
\partial_x f_{z,k}(x^\star_{z,k}) = \underset{x \in [0,\nicefrac{1}{k}]}{\min}\ \partial_x f_{z,k}(x) \leq \partial_x f_{z,k} (\bar{x}_{z,k})
\end{equation*}
Direct computations for $z=3$ and $4$ give
\begin{equation*}
\partial_x f_{3,k}(x^\star_{3,k}) = -\frac{3 (k-1) k}{3 k^2-3 k+1}  > -1,\,
\partial_x f_{4,k}(x^\star_{4,k}) = -\frac{4 (k-1) k}{(1-2 k)^2}  > -1 , \,\forall\, k \geq 2,
\end{equation*}
which concludes this case.

{\bf Lower bound, case $\mathbf{z = 5}$:} here $\partial_x f_{5,k}(x^\star_{5,k}) < -1$ for $k \geq 3$ so we cannot use the same argument. Instead, we prove that $\tilde{x}_{5,k-1} < x^\star_{5,k}$ for $k$ large enough, which give $\partial_x f_{5,k}(\tilde{x}_{5,k-1})$ as a lower bound for $\partial_x f_{5,k}(\bar{x}_{5,k})$. We write
$$x^\star_{5,k} = \frac{\beta_k}{1+k\beta_k},\quad \beta_k = \left(\frac{k}{k-1}\right)^\frac{2}{3}-1 > 0$$
and
$$\tilde{x}_{5,k-1} = \frac{\delta}{k-1},\quad \delta = 1 - \left(\frac{1}{5}\right)^\frac{1}{4} > 0$$
Comparing $x^\star_{5,k}$ and $\tilde{x}_{5,k-1}$
\begin{equation*}
x^\star_{5,k} - \tilde{x}_{5,k-1} = \frac{\beta_k (k-1-\delta k)-\delta}{(k-1)(1+k\beta_k)}
\end{equation*}
The denominator of which is positive. If $k \in \mathbb{C}$ is a zero of the numerator then, straightforward computations give
\begin{equation*}
0 = k^2(k-1-\delta k)^3 - (k-1)^2(\delta+(k-1-\delta k))^3 \in \R_4[X]
\end{equation*}
A numerical computation of the roots of this fourth order polynomial gives 2 complex conjugated roots and 2 real roots. The 2 real roots are approximatively $k_1 = 0.3079371$ and $k_2 = 3.3623924$. So the numerator has constant sign for $k > k_2$ and in particular for $k \geq 4$. Noting that $x^\star_{5,4} > \tilde{x}_{5,3}$ we conclude that $x^\star_{5,k} > \tilde{x}_{5,k-1}$ for all $k \geq 4$ and the following lower bound for the eigenvalue
\begin{equation*}
\partial_x f_{5,k}(\bar{x}_{5,k}) > \partial_x f_{5,k}(\tilde{x}_{5,k-1}),\ \forall k \geq 4
\end{equation*}

A straightforward computation gives that $\partial_x f_{5,k+27}(\tilde{x}_{4,k-1+27}) > -1$ is equivalent to
\begin{eqnarray*}
(6-4\cdot5^{1/4})k^4 + (634+6\sqrt{5}-432\cdot5^{1/4})k^3 + (25126-17496\cdot5^{1/4}486\sqrt{5}-4\cdot5^{3/4})k^2 &\\
+ (442634-314925\cdot5^{1/4}+13122\sqrt{5}-216\cdot5^{3/4})k &\\
+ 2924642-2125764\cdot5^{1/4}+118098\sqrt{5}-2916\cdot5^{3/4} & > 0,
\end{eqnarray*}
which is true for all $k \geq 0$. So $\partial_x f_{5,k}(\tilde{x}_{5,k-1}) > -1$ for all $k \geq 27$. 

For $k \in \llbracket1,26\rrbracket$, a direct numerical computation of the eigenvalues gives that they all belong to  $(-1,0)$.\hfill$\finpreuvem$
\end{proof}

\begin{Rem}
It is actually possible to show that $\partial_x f_{z,k}(\bar{x}_{z,k}) < 0$ for $z \in \{ 3,4,5\}$. To do this, one can compute $\partial_x f_{z,k}(\tilde{x}_{z,k})$ and show that it is decreasing with respect to $k \geq 3$ for $z = 3$ and decreasing for $k \geq 2$ for $z \in \{4,5\}$. This leads to $\partial_x f_{z,k}(\tilde{x}_{z,k}) < 0$ for these $(z,k)$ and as $\tilde{x}_{z,k} < \bar{x}_{z,k}$, $\bar{x}_{z,k} > \hat{x}_{z,k}$. So that $f_{z,k}$ is decreasing at $\bar{x}_{z,k}$.
\end{Rem}

\begin{lem}\label{z6_repel}
For the uniform case, $z=6$ and $k \geq 2$ the fixed point $\bar{x}_{6,k}$ is repelling.
\end{lem}
\begin{proof} 
We prove that $\partial_x f_{6,k}(\bar{x}_{6,k}) < -1$, using the fact that $\partial_x f_{6,k}(\bar{x}_{6,k}) \leq \max ( \partial_x f_{6,k}(\tilde{x}_{z,k}),$  $\partial_ x f_{6,k}(\tilde{x}_{z,k-1}) )$.


As for the case $z=5$, $\partial_x f_{6,k+4}(\tilde{x}_{6,k+4}) < -1$ is equivalent to 
\begin{eqnarray*}
(7-5\cdot6^{1/5})k^5 + (125-75\cdot6^{1/5}-10\cdot6^{2/5})k^4 + (900-450\cdot6^{1/5}-120\cdot6^{2/5}-10\cdot6^{3/5})k^3 &\\
+(3265-1350\cdot6^{1/5}-540\cdot6^{2/5}-90\cdot6^{3/5}-5\cdot6^{4/5})k^2 &\\
+ (5960-2025\cdot6^{1/5}-1080\cdot6^{2/5}-270\cdot6^{3/5}-30\cdot6^{4/5})k & \\
+ (4373-1215\cdot6^{1/5}-810\cdot6^{2/5}-270\cdot6^{3/5}-45\cdot6^{4/5}) & < 0,
\end{eqnarray*}
which is true for $k \geq 0$. So $\partial_x f_{6,k}(\tilde{x}_{6,k}) < -1$ for all $k \geq 4$.

A similar computation for $\partial_x f_{6,k+4}(\tilde{x}_{6,k+3})$ leads to $\partial_x f_{6,k}(\tilde{x}_{6,k-1}) < -1$ for all $k \geq 4$. So $\partial_x f_{6,k}(\bar{x}_{6,k}) < -1$ for $k \geq 4$.


A numerical computation of $\partial_x f_{6,k}(\bar{x}_{6,k})$ for $k = 2$ and $3$ yields values less than $-1$, which concludes the proof.
\end{proof}\hfill$\finpreuvem$\\

\begin{Rem}
Note that as $\partial_x f(x^\star_{z,k}) > -1$ for $z \in \{3,4\}$, the fonction $f_{z,k}$ is contracting on $[\hat{x}_{z,k},\nicefrac{1}{k}]$. This gives the main ingredient to an easier way to prove the convergence of the uniform case than in Proposition \ref{CV_attractive}. However, the given proof of Proposition \ref{CV_attractive} is more general as it only requires the fixed point to be linearly attracting.
\end{Rem}

Lemma \ref{fixed345_attrac} in addition to Proposition \ref{CV_attractive} gives the convergence in the uniform case for $z \in \{3,4,5\}$. The following Proposition extends this result to the non uniform case.  

\begin{prop}\label{nonunif_CV}
Assume that, for the uniform case of the $z$-ary tree, $\bp(n)$ converges to $(\bar x_{z,k},\dots,$ $\bar x_{z,k},1-k\bar x_{z,k})$. Then, in the non uniform case with $\bp_1=\dots=\bp_i>\bp_{i+1}\geq\dots\geq \bp_{k}$, $\bp(n)$ converges to $(\bar x_{z,i},\dots, \bar x_{z,i}, 0_{k-i}, 1-i\bar x_{z,i})$.
\end{prop}

\begin{proof}
Assume that in the uniform case $\bp(n)$ converges. Obviously if $\bp_1=\dots=\bp_i$ and $\bp_j=0$ otherwise, $\bp(n)$ converges to $(\bar{x}_{z,i},\dots, \bar{x}_{z,i}, 0_{k-i}, 1-i\bar{x}_{z,i})$, where $\bar{x}_{z,i}$ is the fixed point of $f_{z,i}$. We want to extend this result to the case $\bp_1=\bp_i>\bp_{i+1}\geq\dots\geq \bp_{k}$.

We introduce the function $\tilde{F} = (\tilde{F}_1,\dots,\tilde{F}_{k-i+1}) = (F_1,F_{i+1},\dots, F_k)$ which is a truncated version of $F$
\[ \tilde{F}_\ell(x)=\left\lbrace\begin{array}{ll}\left(1-(i-1)x_1-\sum_{j=2}^{k-i+1}x_j\right)^z-\left(1-ix_1-\sum_{j=2}^{k-i+1}x_j\right)^z&\mbox{ if $\ell=1$}\\
\left(1-ix_1-\sum_{j=2,j\neq\ell}^{k-i+1}x_j\right)^z-\left(1-ix_1-\sum_{j=2}^{k-i+1}x_j\right)^z&\mbox{ if $\ell = 2,\cdots,k-i+1$,}
\end{array}\right.\]
defined in the set $\mathscr P_{k,i}:=\lbrace (x_1,\dots,x_{k-i+1})\in\mathbb R^{k-i+1}, x_1>x_2\geq x_3\geq\dots \geq x_{k-i+1}>0, ix_1+\sum_{j=2}^{k-i+1}x_j<1 \rbrace$. So, equivalently to our convergence result, we show that, for ${\bf x} \in \mathscr P_{k,i}$, $\tilde{F}^n({\bf x})$ converges to ${\bf \bar{x}_i} = (\bar{x}_{z,i}, 0_{k-i})$.

Let $\varepsilon > 0$ and $\bp=(\bp_1,\dots,\bp_{k-i+1})\in\mathscr P_{k,i}$. From lemma \ref{zeroplus}, we have that, for all $n \in \N$, $\bp_1(n) \in [\alpha,\nicefrac{1}{i}]$. Now, while noting that $\tilde{F}^n (x,0_{k-i}) = (f_{z,i}^n(x),0_{k-i})$, we define $E_n=\lbrace x\in[\alpha,\nicefrac{1}{i}], f_{z,i}^n(x)\in B(\bar{x}_{z,i},\nicefrac{\varepsilon}{2})\rbrace$. Clearly, from the convergence of the uniform case
\[[\alpha,\nicefrac{1}{i}]\subset \cup_{n\geq 0} E_n. \]
As the inverse image of an open set by a continuous fonction, $E_n$ is also an open set for all $n\in\mathbb N$. Since $\cup_{n\geq0}E_n$ is a sequence of open sets covering the compact $[\alpha,\nicefrac{1}{i}]$, there exists $N$ such that  
\[[\alpha,\nicefrac{1}{i}]\subset \cup_{n= 0}^N E_n \]
implying that $\forall x\in[\alpha,\nicefrac{1}{i}], \, \tilde{F}^N(x,0_{k-i})\in B(\bar{x}_{z,i},\nicefrac{\varepsilon}{2}) \times \{0\}^{k-i} \subset B({\bf{\bar{x}_i}},\nicefrac{\varepsilon}{2})$.

On the closed set $\mathscr  G:=[\alpha,\nicefrac{1}{i}]\times \mathbb R_+^{k-i}\cap \overline{\mathscr P_{k,i}}$, $\tilde{F}^N$ is uniformly continuous and thus there exists $\delta>0$ such that 
\[\forall (x,y),(x^\prime,y^\prime)\in\mathscr G, \Vert (x,y)-(x^\prime,y^\prime)\Vert \leq \delta \Rightarrow \Vert \tilde{F}^{N}(x,y)-\tilde{F}^N(x^\prime,y^\prime) \Vert\leq \nicefrac{\varepsilon}{2} \]  
According to lemma \ref{zero},  $\bp_{j}(n)=F^n_j(\bp)\rightarrow 0, \forall j>i$, and consequently there exists $N_1\in\mathbb N$ such that  if $n\geq N_1$, $ \Vert (\bp_{i+1}(n),\dots, \bp_k(n))\Vert\leq \delta$ implying
\[  
\Vert \tilde{F}^{N}(\bp_1(n),\bp_{i+1}(n),\dots,\bp_{k}(n))-\tilde{F}^N(\bp_1(n),0_{k-i}) \Vert\leq \nicefrac{\varepsilon}{2}.\]
Thus, recalling that according to lemma \ref{zeroplus}, $\bp_1(n)\in[\alpha,\nicefrac{1}{i}]$ 
\begin{eqnarray*}
\Vert(\bp_1(n+N),\bp_{i+1}(n+N),\dots,\bp_{k}(n+N))-{\bf{\bar x}}\Vert=\Vert \tilde{F}^{N}(\bp_1(n),\bp_{i+1}(n),\dots,\bp_{k}(n))-{\bf{\bar x}}\Vert\\
\leq\Vert \tilde{F}^{N}(\bp_1(n),\bp_{i+1}(n),\dots,\bp_{k}(n))-\tilde{F}^{N}(\bp_1(n),0_{k-i})\Vert+\Vert \tilde{F}^{N}(\bp_1(n),0_{k-i})-{\bf{\bar x}}\Vert\leq\nicefrac{\varepsilon}{2}+\nicefrac{\varepsilon}{2}=\varepsilon.
\end{eqnarray*}

\end{proof}\hfill$\finpreuvem$

\section{A Specific case of Repelling point and attracting orbit}\label{s_z6_i2} 
In all this section we study the case of a 6-ary tree with 2 dominant diseases. According to lemma \ref{z6_repel}, the convergence  to the fixed point is no longer true and in the present section we prove the existence and uniqueness of an attracting orbit of prime period 2. \\
For the sake of clarity, we write $\bar x$ instead of $\bar x_{6,2}$ and likewise for $f,\, F,\, x^\star,\, \hat x,\,\tilde x$.
Moreover, an easy fact is the following

\begin{Rem}\label{ordre}
If $y$ is a fixed of $f^2$ but not of $f$, then $f(y)$ is a fixed point of $f^2$ distinct of $y$ and
\begin{equation}\label{ordre2}
(f^2)^\prime(y)=(f^2)^\prime(f(y)).
\end{equation} 
\end{Rem}
\begin{figure}[h!]
\caption{$z=6,i=2$, $f$ and $f^2$ } \label{z6i2} 
\includegraphics[width=7cm]{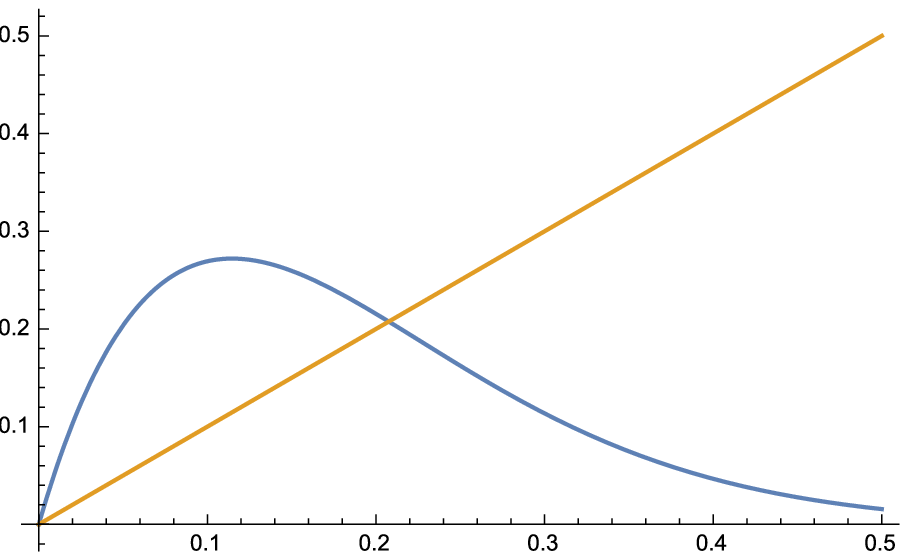}
\includegraphics[width=7cm]{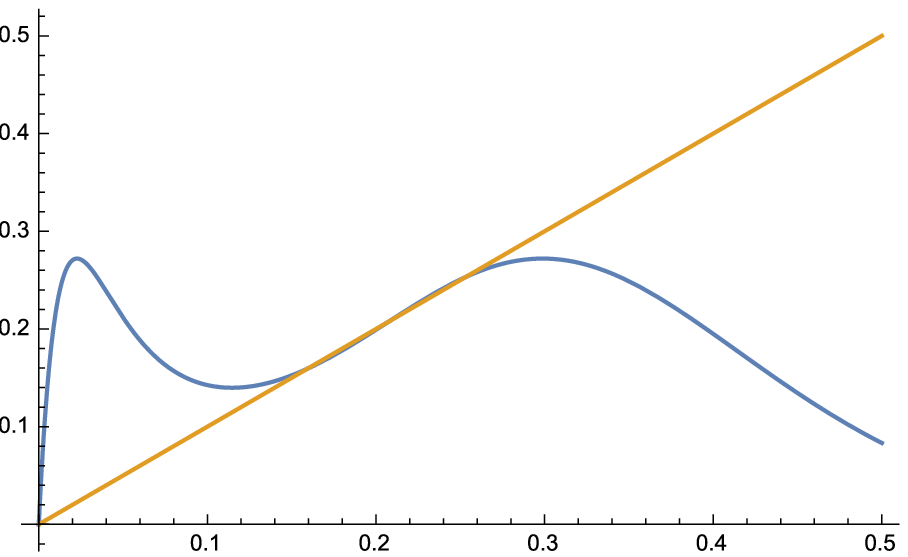}
\end{figure}

\begin{lem}\label{fixe}
All the fixed points of $f^2$ are in $(\hat x, \hat x_r)$ where $\hat x_r:=\sup\lbrace x\in[0,\nicefrac{1}{2}], f(x)=\hat x\rbrace$ and recall that $\hat x={\arg\!\max}_{[0,\nicefrac{1}{2}]}f(x)$. \end{lem}
\begin{figure}[h!]
\caption{Table of variations of $f^2$} \label{TabVarF2}
\vspace*{-0.5cm}
\begin{center}
\includegraphics[width=12cm]{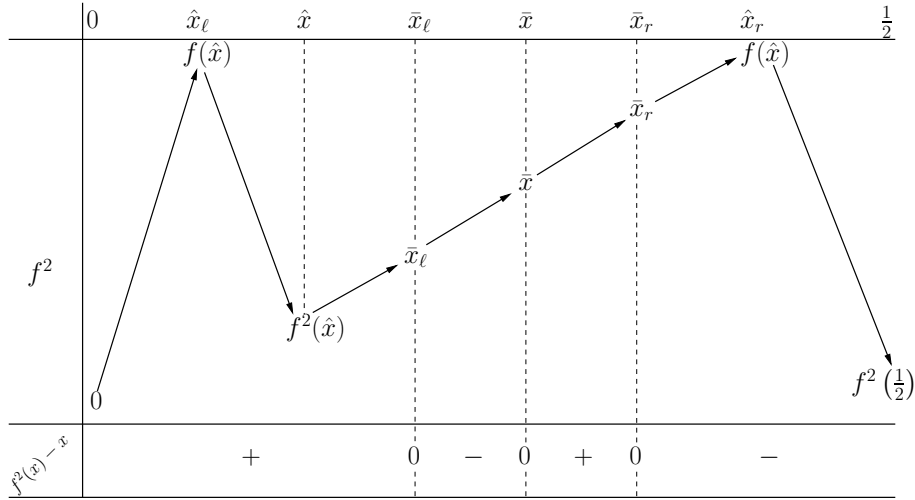}
\end{center}
\end{figure}
\begin{proof}
 First note that $\hat x<\bar x$ 
 and:
\begin{equation}\label{deriv}
f^\prime(x)\geq0\Leftrightarrow x\leq 
\hat x.
\end{equation}
As $\hat x=\arg\!\max_{x\in[0,\nicefrac{1}{2}]}f(x)$, $f([0,\nicefrac{1}{2}])=[0,f(\hat x)]$ implying that $\bar x\in[0,f(\hat x)]$ and 
the existence of  $\hat x_\ell\in[0,\hat x)$ such that $f(\hat x_\ell)=\hat x$ and $f^2(\hat x_\ell)=f(\hat x)=\max_{x\in[0,\nicefrac{1}{2}]} f^2(x).$\\
Moreover note that 
\[f^{\prime\prime}(x)\geq 0\Leftrightarrow x\geq 
x^{\star}, \]
$x^{\star}$ is the only inflexion point of $f$ implying that $f^\prime$ is decreasing on $[0,\hat x]$ as $x^\star>\hat x$.\\
First we show that there is no fixed point on $(0,\hat x)$ and we split this in two cases: $(0,\hat x_\ell]$ and $[\hat x_\ell, \hat x)$.\\
\underline{On $(0,\hat x_\ell]$:}\\ 
Using  \eqref{deriv}, on this interval: 
\begin{equation}\label{positi}
\mathrm{sgn}(f^2(x)-f(x))^\prime=\mathrm{sgn}(f^\prime(x)(f^\prime(f(x))-1))=\mathrm{sgn}(f^\prime(f(x))-1).
\end{equation}
$f^\prime$ being decreasing on $[0,\hat x]$, $f^\prime(f(x))-1$ also on $[0,\hat x_\ell)$, and as $f^\prime(f(0))-1=f^\prime(0)-1=5$ and $f^\prime(f(\hat x_\ell))-1=f^\prime(\hat x)-1=-1$, there exists $\beta\in(0,\hat x_\ell)$ such that $f^2-f$ is non decreasing on $(0,\beta)$ and decreasing otherwise. 
Since $f^2(0)-f(0)=0$ and $f^2(\hat x_\ell)-f(\hat x_\ell)=f(\hat x)-\hat x>0$, $f^2(x)\geq f(x)$ on $[0,\hat x_\ell]$. Thus, as $f(x)>x$ on $(0,\hat x_\ell)$, $f^2$ do not admit a fixed point on this interval.\\
\underline{On $(\hat x_\ell,\hat x)$:}\\
 $f^\prime(x)>0$ and as $f(x)\in(\hat x, f(\hat x)),\, f^\prime(f(x))<0$:
\[\left(f^2(x)-x\right)^\prime=f^\prime(x)f^\prime(f(x))-1<0.\]
Thus $f^2(x)-x$ is decreasing on $(\hat x_\ell,\hat x)$ and, as algebraic computations give $f^2(\hat x)-\hat x>0$, $f^2(x)=x$ admits no solution on $(\hat x_\ell, \hat x]$. \\
In a second time we make a similar reasoning on $[\hat x_r, \nicefrac{1}{2}]$. The existence of $\hat x_r$ is clear noting that $f(\nicefrac{1}{2})=\nicefrac{1}{2^6}<\hat x$.\\
\underline{On $[\hat x_r,\nicefrac{1}{2})$:}\\
$f^\prime(x)<0$ and $f([\hat x_r, \nicefrac{1}{2}])=[f(\nicefrac{1}{2}),\hat x]$. For all $y\in(f(\nicefrac{1}{2}),\hat x)$, $f^\prime(y)>0$, thus:
\[\forall x\in (\hat x_r,\nicefrac{1}{2}), (f^2)^\prime(x)<0.\]
As a result $f^2(x)-x$ is decreasing on $ [\hat x_r,\nicefrac{1}{2})$ and as $f^2(\nicefrac{1}{2})<\nicefrac{1}{2}$, $f^2$ has a unique fixed point on this interval if and only if $f^2(\hat x_r)\geq \hat x_r$.\\
Assume that $f(\hat x_r)\geq \hat x_r$ and denote by $y\in[\hat x_r,\nicefrac{1}{2})$ a fixed point of $f^2$. According to remark \ref{ordre}, $f(y)$ is also a fixed point of $f^2$ (not of $f$) and  
\begin{equation}\label{otherpoint}
(f^2)^\prime(f(y))=(f^2)^\prime(y)<0.
\end{equation}
But as $f^2$ is increasing on $(0,\hat x_\ell)\cup(\hat x,\hat x_r)$ and not increasing elsewhere, \eqref{otherpoint} implies that $f(y)\in (\hat x_\ell, \hat x)$, which is a contradiction as we have proved that there is no fixed point on this interval.\\
Consequently all the fixed points of $f^2$ are on $(\hat x,\hat x_r)$.
\hfill$\finpreuvem$
\begin{figure}[h!]
\caption{$z=6,i=2$ zoom on the fixed points of $f^2$}
\begin{center}
\includegraphics[width=7cm]{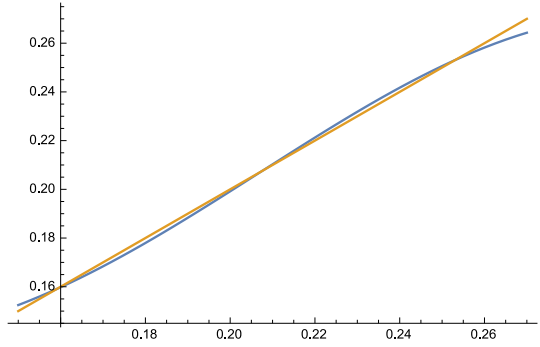}
\end{center}
\end{figure}
\end{proof}

\begin{lem}\label{exact2}
There exists exactly two fixed points of $f^2$ distinct of $\bar x$. They are both attracting.
\end{lem}
\begin{proof}
Noting that $f^\prime(\bar x)<-1$, $\bar x$ is a repelling fixed point of $f^2$ such that
\[(f^2)^\prime(\bar x)=(f^\prime(\bar x))^2>1.\]
So there exists $\varepsilon>0$ such that $f^2(\bar x-\varepsilon)<\bar x-\varepsilon$ and $\bar x+\varepsilon<f^2(\bar x+\varepsilon)$. As $f^2(\hat x)>\hat x$ and $f^2(\hat x_r)<\hat x_r$,  there are at least two additional fixed points on  $(\hat x,\hat x_r)$. Moreover the fact that $f^2$ is increasing on this interval ensures the existence of two attracting fixed point from either side of $\bar x$.\\
We have already seen that the fixed points of $f^2$ are in $[\hat x,\hat x_r]$ and  
moreover a study of the sign of $(f^2)^{\prime\prime}$ shows that $f^2$ admits exactly one inflection point in $[\hat x,\hat x_r]$.
As a result $f^2$ has three fixed points in $(0,\nicefrac{1}{2}]$.
\end{proof}\hfill$\finpreuvem$\\

\begin{prop}\label{Leb}
In the uniform case, if $\lambda$ is the Lebesgue measure on $\mathbb R$:
\[\lambda\left(\left\lbrace \bp_1\in[0,\nicefrac{1}{2}],\, \bp(n)\underset{n\rightarrow+\infty}{\longrightarrow} \lbrace\bar x_\ell, \bar x_r\rbrace\right\rbrace\right)=\nicefrac{1}{2}.\]

\end{prop}

\begin{proof}
According to the lemmas \ref{fixe} and \ref{exact2}, $f^2$ has exactly three fixed points on $[\hat x,\hat x_r]$, denote $\bar x_\ell$ and $\bar x_r$ respectively the one in $(\hat x,\bar x)$ and the one in $(\bar x,\hat x_r)$. \\
\underline{On $(\bar x, \bar x_r)$:}\\
$f^2$ is increasing, $f^2((\bar x, \bar x_r))=(\bar x, \bar x_r)$ and $f^2(x)\geq x$ then $f^2(x_n)=x_{n+2}\geq x_n$. The sequence $(x_{2n})$ is increasing and bounded above by $\bar x_r$, thus it tends to a fixed point of $f^2$ in $(\bar x, \bar x_r]$, $\bar x_r$.\\
\underline{On $[\bar x_r,\hat x_r]$:}\\
 $f^2([\bar x_r,\hat x_r])=[\bar x_r,f(\hat x)]\subset [\bar x_r,\hat x_r)$. Moreover on this interval, $f^2(x)\leq x$ implying that $x_{2n}$ is decreasing and bounded below by $\bar x_r$, thus convergent to $\bar x_r$.\\
\underline{On $(\hat x, \bar x)$:}\\
the reasoning is the same as the one on $(\bar x, \hat x_r)$, and for all $x$ on this interval $x_{2n}\rightarrow \bar x_\ell$.\\
\underline{On $(\hat x_\ell,\hat x)$:}\\
there exists an unique $\bar x_1$ such that $f(\bar x_1)=\bar x$. Note that on this interval $f^2$ is decreasing implying that 
 \begin{itemize}
 \item if $x\in[\hat x_\ell, \bar x_1)$, $f^2(x)\in(\bar x, f(\hat x)]\subset (\bar x, \hat x_r)$ and $x_{2n}\rightarrow\bar x_r$;
 \item if $x\in(\bar x_1, \hat x]$, $f^2(x)\in[f^2(\hat x),\bar x)\subset (\hat x, \bar x)$ and then $x_{2n}\rightarrow\bar x_\ell$.
\end{itemize}
\underline{On $[\hat x_r,\nicefrac{1}{2}]$:}\\
there exists $\bar x_2$ such that $f(\bar x_2)=\bar x$. On this interval $f^2$ is decreasing and 
\[f^2([\hat x_r, \bar x_2))=(\bar x, f(\hat x)]\mbox{ and }f^2((\bar x_2, \nicefrac{1}{2}])=[f^2(\nicefrac{1}{2}),\bar x) \]
If $x\in[\hat x_r, \bar x_2)$, $f^2(x)\in(\bar x, f(\hat x)]$ and we are in one of the previous cases. Let us postpone the other case ($x \in (\bar{x}_2,\nicefrac{1}{2}]$).\\
\underline{On $[0,\hat x_\ell]$:}\\
 $f^2$ is increasing implying the existence of $\gamma$ such that $f^2(\gamma)=\hat x_\ell$. If $x\in[\gamma, \hat x_\ell]$, $f^2(x)\in[\hat x_\ell, f(\hat x)]$ and we are in one of the previous cases.\\
We can easily show by contradiction that for all $x\in[0,\gamma]$ there exists $n$ such that $f^{2n}(x)\in[\hat x_\ell, f(\hat x)]$. If it is not the case $f^{2n}(x)$ is an increasing sequence bounded above, and as a consequence it converges to a fixed point of $f^2$ in $[0,\hat{x}_\ell]$. According to lemma \ref{fixe}, there is no such point. There is however countably many points that converges to $\bar{x}$ (because $f^2([0,\gamma]) \subset [0,f(\hat{x})]$).\\
To conclude if $x\in(\bar x_2,\nicefrac{1}{2}]$, $f^2(x) \in [f^2(\nicefrac{1}{2}),\bar{x})$ and we are in one of the previous cases.
\end{proof}\hfill$\finpreuvem$\\

To obtain results for the non uniform case, we need the following:
\begin{lem}
$\lbrace (\bar x_\ell,\bar x_\ell, 0_{k-2},1-2\bar x_\ell), (\bar x_r,\bar x_r,0_{k-2},1-2\bar x_r)\rbrace$ is an attracting orbit for $F$ in $\mathbb R^{k+1}$.
\end{lem}

\begin{proof}
Note that it is the same to prove that $\lbrace (\bar x_\ell, 0_{k-2}), (\bar x_r,0_{k-2})\rbrace$ is an attracting orbit for  $\tilde F=(\tilde F_1,\dots,\tilde F_{k-1})$ where
\[ \tilde F_\ell(x)=\left\lbrace\begin{array}{ll}\left(1-x_1-\sum_{j=2}^{k-1}x_j\right)^z-\left(1-2x_1-\sum_{j=2}^{k-1}x_j\right)^z&\mbox{ if $\ell=1$}\\
\left(1-2x_1-\sum_{j=2,j\neq\ell}^{k-1}x_j\right)^z-\left(1-2x_1-\sum_{j=2}^{k-1}x_j\right)^z&\mbox{ otherwise.}
\end{array}\right.\]
Let  $B_{i,j}(x)=\frac{(\partial \tilde F^2)_i}{\partial x_j}(x)=\sum_{m=1}^{k-1}\frac{\partial\tilde F_i}{\partial x_m}(\tilde F(x))\frac{\partial\tilde F_m}{\partial x_j}(x) $ and note that:
\[B_{i,j}(\bar x_\ell, 0_{k-2})=B_{i,j}(\bar x_r, 0_{k-2})=\left\lbrace\begin{array}{cc}
0&\mbox{ if $i\neq j$}\\
\frac{\partial \tilde F_i}{\partial x_i}(\bar x_\ell,0_{k-2})\frac{\partial\tilde F_i}{\partial x_i}(\bar x_r,0_{k-2})&\mbox{ otherwise. }
\end{array}
 \right.\]
$B_{i,j}(\bar x_\ell, 0_{k-2})$ is upper triangular and   $|B_{1,1}(\bar x_\ell,0_{k-2})|=\left|\frac{\partial\tilde F_1}{\partial x_1}(\bar x_\ell,0_{k-2})\frac{\partial\tilde F_1}{\partial x_1}(\bar x_r,0_{k-2})\right|<1$ according to lemma \ref{exact2}. 
It remains to prove that the other eigenvalues of $B$ are (strictly) bounded above by 1, the positivity being obvious.\\ 
With the same notations as the one of the proof of lemma \ref{framing}, recall that if $x<\tilde x$, $\frac{\partial\tilde F_i}{\partial x_i}(x,0_{k-2})<1$ and greater than 1 otherwise, and one can easily check that  $f^2(\tilde x)\geq \tilde x$. Then, as $\tilde x<\bar x$, obviously $\tilde x\leq \bar x_{\ell}$  and consequently $\frac{\partial\tilde F_i}{\partial x_i}(\bar x_\ell,0_{k-2})<1$ and as $\tilde x< \bar x_{r}$, $\frac{\partial \tilde F_i}{\partial x_i}(\bar x_\ell,0_{k-2})\frac{\partial\tilde F_i}{\partial x_i}(\bar x_r,0_{k-2})<1$. 
\end{proof}\hfill$\finpreuvem$\\

Now we are able to prove Theorem \ref{concl}. The previous lemma ensures the existence of $\varepsilon>0$ such that if $x\in B((\bar x_r,0_{k-2}), \varepsilon)$ (respectively  $B((\bar x_\ell,0_{k-2}), \varepsilon)$), then $\lim_{n\rightarrow+\infty}\tilde F^{2n}(x)=(\bar x_r, 0_{k-2})$ (respectively $(\bar x_\ell, 0_{k-2})$).\\
According to the stable manifold Theorem, the stable manifold of $\tilde F$, $W^s(\bar{x},0_{k-2})$ is a one dimensional smooth manifold in a neighborhood of $(\bar{x},0_{k-2})$. So there exists $\varepsilon^\prime$ such that $W_{\varepsilon^\prime}^s = W^s(\bar{x},0_{k-2}) \cap B((\bar{x},0_{k-2}),\varepsilon^\prime)>0$ is negligible with respect to the Lebesgue measure. In other words, for almost every $x \in B((\bar{x},0_{k-2}),\varepsilon^\prime)$, there exists $m > 0$ such that $\tilde F^{m}(x)\notin  B((\bar x,0_{k-2}), \varepsilon^\prime)$.\\
 Let $E_n=\lbrace x_1\in[\bar x+\varepsilon^\prime, \hat x_r], \tilde F^{2n}(x_1,0_{k-2})\in B((\bar x_r,0_{k-2}),\nicefrac{\varepsilon}{2}) \rbrace$, and according to Proposition \ref{Leb}, we have already seen that $[\bar x+\varepsilon^\prime, \hat x_r]\subset \cup_{n\geq 0} E_n$. With a similar reasoning as the one of Proposition \ref{nonunif_CV},  there exists $N\geq0$ such that $[\bar x+\varepsilon^\prime, \hat x_r]\subset \cup_{n=0}^N E_n$ implying that for all  $x_1\in[\bar x+\varepsilon^\prime, \hat x_r]$, $\tilde F^{2N}((x_1,0_{k-2}))\in B((\bar x_r,0_{k-2}),\nicefrac{\varepsilon}{2})$.\\
Using the uniform continuity of $\tilde F^{2N}$ on the compact set $\mathscr H:=[\bar x+\varepsilon^\prime,\hat x_r ]\times \mathbb R^{k-2}\cap\overline{\mathscr P_{k,2}}$, there exists $\nu>0$ such that 
\[\forall x,y\in\mathscr H,  \Vert x-y\Vert\leq \nu \Rightarrow \Vert \tilde F^{2N}(x)-\tilde F^{2N}(y)\Vert\leq \nicefrac{\varepsilon}{2}. \]
Thus for $x=(x_1,\dots,x_{k-1})\in [\bar x+\varepsilon^\prime, \hat x_r]\times B(0_{k-2}, \nu)$:
\begin{eqnarray*}
\Vert \tilde F^{2N}(x_1,\dots,x_{k-1})-(\bar x_r,0_{k-2})\Vert &\leq& \Vert \tilde F^{2N}(x_1,\dots,x_{k-1})-\tilde F^{2N}(x_1,0_{k-2})\Vert\\
&\,&\quad+\Vert\tilde  F^{2N}(x_1,0_{k-2})-(\bar x_r,0_{k-2})\Vert\\
&\leq& \nicefrac{\varepsilon}{2}+\nicefrac{\varepsilon}{2}=\varepsilon.
\end{eqnarray*}
Consequently $\tilde F^{2N}(x)\in B((\bar x_r,0_{k-2}), \varepsilon)$ implying that $F^{2n}(x)\rightarrow\bar x_r$.\\
Note that we can realize exactly the same reasoning on $[\hat x,\bar x-\varepsilon^\prime]$.\\
Let $x_1\in[\alpha,\nicefrac{1}{2}]\backslash[\hat x, \hat x_r]$ (see Lemma \ref{zeroplus} for the definition of $\alpha$) and according to Proposition \ref{Leb} there exists $n$ such that $\tilde F^{2n}(x_1, 0_{k-1})\in [\hat x,\hat x_r]\times B(0_{k-2}, \nu)$. 
Using again the uniform continuity of $\tilde F^{2n}$, there exists $\nu^\prime$ such that  
\[\Vert x-y\Vert \leq \nu^\prime\Rightarrow \Vert\tilde  F^{2n}(x)-\tilde F^{2n}(y)\Vert\leq \delta:=\min\left(\nu,\frac{\tilde F^{2n}_1(x_1, 0_{k-2})-\hat x}{2}, \frac{\hat x_r-\tilde F^{2n}_1(x_1,0_{k-2})}{2}\right) \]
implying that if $\Vert (x_2,\dots,x_{k-1})\Vert\leq  \nu^\prime$ then 
\begin{eqnarray*}
\frac{\hat x+\tilde F^{2n}_1(x_1,0_{k-2})}{2}\leq &\tilde F^{2n}_1(x_1,\dots,x_{k-1})&\leq \frac{\hat x_r+\tilde F^{2n}_1(x_1,0_{k-2})}{2}\\
0\leq &\tilde F^{2n}_j(x_1,\dots,x_{k-1})&\leq \nu, \forall j>1
\end{eqnarray*}
and thus $\tilde F^{2n}(x_1,\dots,x_{k-1})\in [\hat x,\hat x_r]\times B(0_{k-2}, \nu) $.\\
Note that we write $\tilde \bp=(\tilde \bp_1,\dots,\tilde \bp_{k-1})$ for $(\bp_1,\bp_3,\bp_4, \dots,\bp_{k})$ and $\tilde \bp(n)=\tilde F^n(\tilde \bp)$.\\
According to lemma \ref{zeroplus}, $\tilde \bp_{j}(n)\rightarrow 0$ for all $j\geq 2$ and according to lemma \ref{zero}, $\tilde \bp_1(n) \in[\alpha, \nicefrac{1}{2}]$ for all $n\geq1$.  There exists $N_1$ such that $\forall n\geq N_1$, $\tilde \bp(n)\in (0,\nicefrac{1}{2}]\times B({0_{k-2}},\nu^\prime)$. Then there exists $n\geq0$ such that 
\[\tilde \bp(N_1+2n)\in[\hat x,\hat x_r]\times B(0_{k-2}, \nu). \]
If $\tilde \bp(N_1+2n)\in [\hat x,\bar x-\varepsilon^\prime]\cup [\bar x+\varepsilon^\prime,\hat x_r ]\times B(0_{k-2}, \nu)$, we have the convergence to $(\bar x_\ell,0_{k-2})$ or $(\bar x_r,0_{k-2})$. 
Otherwise $\tilde \bp(N_1+2n)\in B(\bar x,\varepsilon^\prime)\times B(0_{k-2},\nu)$ and 
\begin{itemize}
\item either $\lim_{n\rightarrow+\infty}\tilde \bp(2n)=(\bar x, 0_{k-2})$,
\item or there exists $m>n$ such that $\tilde \bp(N_1+2m)\in  [\hat x,\bar x-\varepsilon^\prime]\cup [\bar x+\varepsilon^\prime,\hat x_r ] \times B(0_{k-2},\nu)$ and we have convergence to $(\bar x_\ell,0_{k-2})$ or $(\bar x_r,0_{k-2})$. 
\end{itemize}
If we are in the first case, it means that $\tilde \bp\in W_s=\cup_{n>0}\tilde F^{-n}(W^s_{\varepsilon^\prime})$ and as a countable union of negligible sets with respect to the Lebesgue measure, $W_s$ is also negligible: 
\begin{equation}
\lambda\left(\tilde \bp\in \mathscr P_{k,2}, \lim_{n\rightarrow+\infty}\tilde \bp(2n)\notin\lbrace(\bar x_\ell,0_{k-2}), (\bar x_r,0_{k-2})\rbrace \right)=0
\end{equation}



\begin{Rem}
This result is not limited to the case $z=6$ and $i=2$. Indeed, according to our proof, the only conditions a case $(z,i)$ has to satisfy are~:
\begin{enumerate}
\item
	$\bar{x}_{z,i}$ is such that $\partial_x f_{z,i}(\bar{x}_{z,i}) < -1$.
\item
	$f_{z,i}(\hat{x}_{z,i}) > \hat{x}_{z,i}$.
\item
	$f_{z,i}(\nicefrac{1}{i}) < \hat{x}_{z,i}$.
\end{enumerate}
The last two conditions appear in the proof of lemmas \ref{fixe} and \ref{exact2}. While condition 3 seems to be always satisfied, condition 2 is not. For instance, it is not satisfied for $(z,i) = (12,2)$ and numerical simulations show the existence of an orbit of period 4 (see Figure \ref{z12i2}) for this case. However, the aim of the paper is not the study of all non convergent cases.
\end{Rem}

\begin{figure}[h!]
\caption{$z=12,i=2$, $f$ (red), $f^2$ (yellow), $f^4$ (blue) } \label{z12i2}
\begin{center}
\includegraphics[width=7cm]{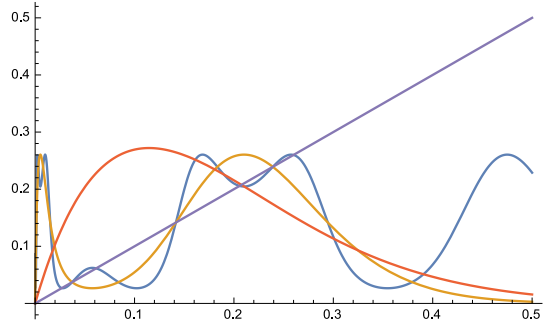}
\end{center}
\end{figure}


\section{Open questions and variant case}
In \cite{Benjamini}, the authors study a binary tree with the following changes 
\begin{itemize}
\item $k=2$, we have only two diseases; 
\item  (R3) is replaced by  (R4): if only one of the leaves is infected, the node is infected by it with probability $\alpha$ and not infected with probability $1-\alpha$.
\end{itemize}
The authors obtain the following
\begin{The}
For all $\bp\in\mathscr P_2$, $\bp(n)$ converges and 
\begin{enumerate}
\item If $\alpha>\nicefrac{1}{2}$
\[\lim_{n\rightarrow\infty}\bp(n)=\left\lbrace\begin{array}{ll}
(1,0,0)&\mbox{ if $\bp_1>\bp_2$},\\
\left(\frac{2\alpha-1}{4\alpha-1},\frac{2\alpha-1}{4\alpha-1},\frac{1}{4\alpha-1}\right)&\mbox{ if $\bp_1=\bp_2$.}
\end{array}\right.\] 
\item If $\alpha=\nicefrac{1}{2}$
\[\lim_{n\rightarrow\infty}\bp(n)=\left\lbrace\begin{array}{ll}
(\bp_1-\bp_2, 0,1-\bp_1+\bp_2)&\mbox{ if $\bp_1>\bp_2$},\\
\left(0,0,1\right)&\mbox{ if $\bp_1=\bp_2$.}
\end{array}\right.\] 
\item If $\alpha<\nicefrac{1}{2}$
$\lim_{n\rightarrow\infty}\bp(n)=\left(0,0,1\right)$.
\end{enumerate}
\end{The}
In the same spirit, we replace (R3') by (R4'): 
\begin{eqnarray*}
\p\left(\bigotimes_{j=1}^z X_j=\1\middle|\mathscr A_i^\ell(z) \right)&=&(1-\alpha_i)^{z-\ell}
\end{eqnarray*} 
where $\alpha_i$ is the probability of infection $i$ and $\mathscr A_i^\ell(z):=\lbrace S_z=\lbrace1,\dots,\ell\rbrace,X_{\ell+1}=\dots=X_z=e_i\rbrace$ the event {\it{the $\ell$ firsts children are not sick  and the others are infected by disease $i$}}.\\
Recall (R1'): if all the children have the same state (infected or not) the ancestor is infected (or not) by it (which does not seem natural). 
\begin{figure}[h!]
\caption{(R4') for $N=3$ a.s.} \label{fe4}
\begin{center}
\includegraphics[width=8cm]{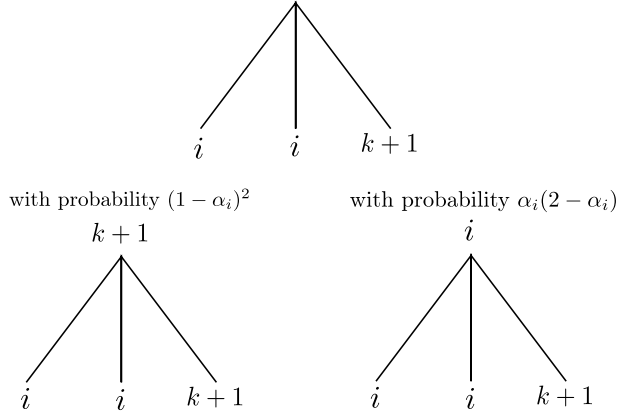}
\end{center}
\end{figure}

We obtain the equivalent of lemma \ref{rec}:
\begin{lem}
For all $n\geq1$:
\begin{equation*}
\bp_{i}(n+1)=\left\lbrace\begin{array}{ll}
G_N(\bp_{k+1}(n)+\bp_i(n))-G_N(\bp_{k+1}(n)+\bp_i(n)(1-\alpha_i))+G_N((1-\alpha_i)\bp_i(n)), \mbox{if $i\leq k$ }\\
1-\sum_{j=1}^k\bp_{i}(n+1)\mbox{ otherwise.}
\end{array}\right.
\end{equation*} 
\end{lem}
\begin{proof}
For $i\neq k+1$, the recursion formula obtained here is:
\begin{align*}
\bp_i(n+1)&=\sum_{z=2}^\infty q_z\sum_{\ell=0}^{z-1}\p\left(|S_z|=\ell, \bigotimes_{j=1}^zX_j=e_i    \right)\\
&=\sum_{z=2}^\infty q_z\sum_{\ell=0}^{z-1}(_\ell^z)\bp_{k+1}^\ell(n) \bp_{i}^{z-\ell}(n)\p\left( \bigotimes_{j=1}^zX_j=e_i\left\vert \mathscr A_i^\ell(z)\right.\right)\\
&=\sum_{z=2}^\infty q_z\left[\sum_{\ell=1}^{z-1}(_\ell^z)\bp_{k+1}^\ell(n) \bp_{i}^{z-\ell}(n)(1-(1-\alpha_i)^{z-\ell})+\bp_i(n)^z\right]\\
&=\sum_{z=2}^\infty q_z\left((\bp_{k+1}(n)+\bp_i(n))^z-(\bp_{k+1}(n)+\bp_i(n)(1-\alpha_i))^z+(1-\alpha_i)^z\bp_i(n)^z\right)\\
&=G_N(\bp_{k+1}(n)+\bp_i(n))-G_N(\bp_{k+1}(n)+\bp_i(n)(1-\alpha_i))+G_N((1-\alpha_i)\bp_i(n)).\hfill\finpreuvem
\end{align*}
\end{proof}

In all the following we assume that 
\begin{equation}\label{alpha}
\forall i\in\llbracket 1,k\rrbracket, \alpha_i=\alpha> \E[N]^{-1}. 
\end{equation}
\begin{lem}\label{zerobis}
There exists $\eta>0$ such that 
\[\forall n\in \mathbb N, \bp_1(n)\geq \min \lbrace G_N(\eta),\bp_1\rbrace\]
\end{lem}

\begin{proof}
According to \eqref{alpha}, there exists $\delta>0$ such that $\alpha=\E[N]^{-1}+\delta$.
As $G_N(1-x)=G_N(1)-xG_N^\prime(1)+\varepsilon(x)$ where $\nicefrac{\varepsilon(x)}{x}\rightarrow 0$ when $x$ goes to 0, there exists $\eta>0$ small enough such that if $0<x\leq\eta$, $|\varepsilon(x)|\leq \nicefrac{\delta x G_N^\prime(1)}{k}$. Then for $0<x\leq \nicefrac{\eta}{k}$ and $0<y<(k-i)x$: 
\begin{eqnarray*}
{G_N\left(1-(i-1)x-y\right)-G_N\left(1-(i-1+\alpha) x-y\right)}
\geq{\alpha x}G^\prime_N(1)-\delta xG^\prime_N(1)\geq x,
\end{eqnarray*}
and as  $G_N((1-\alpha)x)\geq0$: 
\[G_N(1-(i-1)x-y)-G_N(1-(i-1+\alpha) x-y)+G_N((1-\alpha)x)\geq x.\]
The rest of the proof is exactly the same as the one of lemma \ref{zeroplus}.
\end{proof}\hfill$\finpreuvem$
\begin{lem}
For all $j>i$, $\bp_j(n)\underset{n\rightarrow\infty}{\rightarrow}0$. 
\end{lem}
\begin{proof}
Writing for all $1\leq j\leq k$:
\[\bp_{j}(n+1)=\alpha\sum_{z\geq2}q_z \bp_{j}(n)\sum_{\ell=0}^{z-1}(\bp_{j}(n)+\bp_{k+1}(n))^{z-1-\ell} (\bp_{k+1}(n)+(1-\alpha)\bp_{j})^\ell+(1-\alpha)^z\bp_{i+1}(n)^{z}\]
it is not difficult to see that the sequence $w^\prime_n=\frac{\bp_{i+1}(n)}{\bp_1(n)}$ is positive and decreasing. 
The rest of the proof is the same as the one of Lemma \ref{zero}. 
\end{proof}\hfill$\finpreuvem$
\begin{prop}
If $\bp_1>\bp_2\geq\dots\geq \bp_{k}$ then $\lim_{n\rightarrow+\infty}\bp_1(n)=1$. 
\end{prop}

\begin{proof}
We have already shown that $\bp_{j}(n)\underset{n\rightarrow\infty}{\rightarrow}0$ for all $j>1$ and according to lemma \ref{zerobis} $\liminf_{n\rightarrow\infty}\bp_1(n)>0$. Consequently 
\begin{eqnarray*}
\bp_{1}(n+1)&=&G_N\left(1-\sum_{j=2}^k\bp_j(n)\right)-G_N\left(1-\sum_{j=2}^k\bp_j(n)-\alpha \bp_{1}(n)\right)+G_N\left((1-\alpha)\bp_1(n)\right)\\
\liminf_{n\rightarrow\infty}\bp_{1}(n+1)&=&G_N(1)-G_N\left(1-\alpha\liminf_{n\rightarrow\infty}\bp_{1}(n)\right)+G_N\left((1-\alpha)\liminf_{n\rightarrow\infty}\bp_1(n)\right)\\
&=&1-G_N\left(1-\alpha\liminf_{n\rightarrow\infty}\bp_{1}(n)\right)+G_N\left((1-\alpha)\liminf_{n\rightarrow\infty}\bp_1(n)\right).
\end{eqnarray*} 
Thus, $\liminf_{n\rightarrow\infty}\bp_{1}(n)$ is a fixed point of $x\mapsto1-G_N(1-\alpha x)+G_N((1-\alpha)x)$ on $(0,1]$. This function being increasing, the solution is 1. As a result $\liminf_{n\rightarrow\infty}\bp_{1}(n)=\lim_{n\rightarrow\infty}\bp_1(n)=1$. 

\end{proof}\hfill$\finpreuvem$\\

Here we give some open questions that may be interesting to study. 

\begin{enumerate}
\item What happens if we slightly change (R1') and (R2'): although all children are sick, we need at least a disease to be expressed. In a more probabilistic way, for all $0\leq \ell\leq z-1$
\begin{eqnarray*}
\p\left(\bigotimes_{j=1}^z X_j=e_i \middle| \mathscr A_i^\ell(z) \right)&=&\sum_{w=1}^{z-\ell}\alpha_i^w(1-\alpha_i)^{z-\ell-w}.
\end{eqnarray*}
Assuming that $\alpha_i<\E[N]^{-1}$:
\begin{equation}\label{perco}
\lim_{n\rightarrow+\infty}\bp_i(n)=0.
\end{equation}
Indeed, we can bound above $\lim_{n\rightarrow+\infty}\bp_i(n)$ by the probability that in a GW with reproduction law $N$ and percolation probability $\alpha_i$, there is an infinite connected component containing $\phi$. It is well known that $\alpha_c=\E[N]^{-1}$, is the critical probability for the existence of such a connected component which gives \eqref{perco} 
\item An interesting subject is the link between the law of $N$ and the existence of a fixed point. It seems obvious that if $N$ is heavy tailed, there is no fixed point but what happens if $\p(N\geq 6)$ is very low? Is there a critical value for this probability implying the existence or not of a fixed point? 
\item Numerical simulations suggest the existence of a unique attracting orbit for every $z$ and $k$. The study of the Schwarzian derivative of $f$ could be a good point of view. Moreover for every $p\in \mathbb N$, it seems that we can find $z$ and $k$ such that the attracting orbit has $2p$ for prime period.  To conclude, can we write a complete classification of the map with topological conjugacy: Two maps $f:A\rightarrow A$ and $g:B\rightarrow B$ are said to be topologically conjugate if there exists a homeomorphism $h:A\rightarrow B$ such that $f=h^{-1}\circ g\circ h$. In this case $f$ and $g$ are equivalent in terms of their dynamics. For instance if $x$ is a fixed point of $f$, $h(x)$ is a fixed point of $g$. 
\end{enumerate}


\section{Appendix}

In this section, we remind elementary definitions and results on discrete dynamical systems we use in section \ref{s_zary} and \ref{s_z6_i2}. For a more thorough presentation, we refer the reader to \cite{Deva,Teschl}.

So let $M$ be a metric space and $f:M\to M$ a $C^1$ mapping. We are interested in the discrete dynamical system corresponding to
\begin{equation}
f^0(x) = x,\ f^n(x) = f^{n-1}\circ f(x),\ n \in \N,\ x \in M
\end{equation}
\begin{Def}[fixed and periodic points]
A point $y \in M$ satisfying $y=f(y)$ is called a fixed point of $f$.\\
A point $y \in M$ satisfying $y = f^p(y)$ for $p \in \N^\star$ and $y\neq f^n(y)$ for $n \in \{1,\cdots, p-1\}$ is called a periodic point of $f$ of prime period $p$.
\end{Def}

Note that a fixed point of $f$ is a periodic point of $f$ of prime period $1$.

\begin{Def}[attracting and repelling periodic points]
Let $y \in M$ be a periodic point of $f$ of prime period $p$.
\begin{enumerate}
\item
	$y$ is an attracting periodic point of $f$ (of prime period $p$) if there exists an open neighborhood $U$ of $y$ such that
	\begin{equation*}
	\forall x \in U,\ \lim_{n\rightarrow+\infty} f^{np}(x) = y.
	\end{equation*}
\item
	$y$ is a repelling periodic point of $f$ (of prime period $p$) if there exists an open neighborhood $U$ of $y$ such that
	\begin{equation*}
	\forall x \in U\backslash \lbrace y\rbrace,\ \exists n \in \N\ s.t.\ f^{np}(x) \notin U.
	\end{equation*}
\end{enumerate}
\end{Def}

\begin{Def}[hyperbolic periodic point]
Let $y \in M$ be a periodic point of $f$ of prime period $p$, and let $A = D(f^p)(y)$ be the Jacobian matrix of $f^p$ at $y$. If $A$ is invertible and has none of its eigenvalues on the unit circle, we call $y$ a hyperbolic periodic point of $f$ (of prime period $p$).
\end{Def}

\begin{prop}
Let $y$ be an hyperbolic periodic point of $f$ of prime period $p$, and let $A = D(f^p)(y)$ be the Jacobian matrix of $f^p$ at $y$.
\begin{enumerate}
\item
	If all the eigenvalues of $A$ are inside the unit circle, then $y$ is an attracting periodic point of $f$. $y$ is said to be linearly attracting. 
\item
	If all the eigenvalues of $A$ are outside the unit circle, then $y$ is a repelling periodic point of $f$.
\end{enumerate}
\end{prop}

Note that a hyperbolic periodic point which is neither attracting nor repelling is called a saddle point. In this case, it is of interest to define the stable and unstable sets of the mapping $f^p$.

\begin{Def}[stable and unstable sets]
Let $y$ be a hyperbolic periodic point of $f$ of prime period $p$. We define the stable and unstable sets $W^s(y)$ and $W^u(y)$ as 
\begin{eqnarray}
W^s(y)&:=&\left\lbrace x\in M\ |\lim_{n\rightarrow+\infty}f^{np}(x)=y\right\rbrace\\
W^u(y)&:=&\left\lbrace x\in M\ |\lim_{n\rightarrow-\infty}f^{np}(x)=y\right\rbrace.
\end{eqnarray} 
\end{Def}

Note that if $y$ is a attracting periodic point of $f$, then  there exists an open neighborhood $U$ of $y$ such that $W^u(y)\cap U = \lbrace y\rbrace$ and $W^s(y)\cap U = U$. Similarly, if $y$ is repelling, then there exists an open neighborhood $U$ of $y$ such that $W^s(y)\cap U = \lbrace y\rbrace$ and $W^u(y)\cap U = U$.

\begin{The}[Hartman-Grobman Theorem]
Let $y$ be a hyperbolic fixed point of $f$. Then there exists a neighborhood $U$ of $y$ and a homeomorphism $h:U \to M$ such that
\begin{equation*}
f_{|U} = h^{-1}\circ A \circ h,
\end{equation*}
where $A$ is the Jacobian matrix of $f$ at $y$.
\end{The}

This theorem states that in a neighborhood of a hyperbolic fixed point, $f$ is topologically conjugated to its linearization. So in this neighborhood, the behavior of the dynamical system is qualitatively the same as the one of its linearization. This leads to the stable manifold theorem.
\begin{The}
Let $y$ be a hyperbolic fixed point of $f$. Then
\begin{enumerate}
\item
	$W^s(y)$ is a smooth manifold and its tangent space at $y$ is the stable space of the linearization of $f$ at $y$.
\item
	$W^u(y)$ is a smooth manifold and its tangent space at $y$ is the unstable space of the linearization of $f$ at $y$.
\end{enumerate}
\end{The}

\end{document}